\documentclass[reqno, 11pt]{amsart}
\usepackage{srcltx}
\usepackage{amssymb,amsfonts, amsmath}
\usepackage{eucal}
\usepackage{graphicx}
\usepackage{mathrsfs}
\usepackage{color}
\usepackage{dsfont}
\usepackage[normalem]{ulem} 
\newtheorem{theorem}{Theorem}[section]

\newtheorem{lemma}{Lemma}[section]
\newtheorem{proposition}[theorem]{Proposition}
\theoremstyle{definition}

\theoremstyle{remark}

\numberwithin{equation}{section}

\newcommand{\E}{\mathsf{E}}
\newcommand{\Q}{\mathsf{Q}}

\DeclareMathOperator*{\sign}{\mathrm{sign}}
\DeclareMathOperator*{\trace}{\mathrm{trace}}

\newcommand{\lef}{\langle\hskip -1.8pt \langle}
\newcommand{\rig}{\rangle\hskip -1.8pt \rangle}

\theoremstyle{definition}

\newtheorem{remark}{Remark}
\newtheorem{example}{Example}[section]

\definecolor{gray}{rgb}{0.9,0.9,0.9}

\input cyracc.def

\begin{document}

\title[] {When A Stochastic Exponential is a true Martingale. Extension of a method of Bene\^s.}

\author{F. Klebaner}
\address{School of Mathematical Sciences\\  Monash
University, Building 28\\ Clayton Campus, Wellington road, Victoria
3800\\ Australia} \email{fima.klebaner@sci.monash.edu.au}
 \thanks{\bf Acknowledgement.
 Research   supported by the Australian Research
Council Grants DP0881011 and DP0988483.}

\author{R. Liptser}
\address{Department of Electrical Engineering Systems,
Tel Aviv University, 69978 Tel Aviv, Israel}
\email{liptser@eng.tau.ac.il; rliptser@gmail.com}



\maketitle
\begin{abstract}
Let $\mathfrak{z}$ be a stochastic exponential, i.e., $\mathfrak{z}_t=1+\int_0^t\mathfrak{z}_{s-}dM_s$, of a local martingale $M$ with jumps   $\triangle M_t>-1$.  Then  $\mathfrak{z}$ is a nonnegative local martingale with $\E\mathfrak{z}_t\le 1$. If $\E\mathfrak{z}_T= 1$, then  $\mathfrak{z}$ is a martingale on the time interval $[0,T]$. Martingale property plays an important role in many
applications. It is therefore of interest to give natural and easy verifiable conditions for the martingale property.
In this paper, the property $\E\mathfrak{z}_{_T}=1$  is verified with the so-called  linear growth
 conditions  involved in the definition of parameters of $M$,   proposed by Girsanov
\cite{Girs}. These conditions generalize  the Bene\^s idea, \cite{Benes}, and avoid the technology of
piece-wise approximation.
These conditions are applicable even if Novikov, \cite{Novikov}, and Kazamaki, \cite{Kaz}, conditions
fail.  They are effective  for Markov processes that explode, Markov processes with jumps  and also non Markov processes.
Our approach is different  to   recently published papers
\cite{CFY} and \cite{MiUr}.
\end{abstract}

\section{\bf Introduction}
\label{sec-1}
Let  $M=(M_t)_{t\in[0,T]}$ be a martingale (local martingale)
with paths from Skorokhod's space $\mathbb{D}$. So
$M=M^c+M^d$, where  $M^c$ and $M^d$ are continuous and purely discontinuous martingales respectively.
Denote  by $\triangle M_t:=M_{t}-M_{t-}$ the jump process of the martingale $M$ and by
$\langle M^c\rangle_t$ the predictable quadratic variation of continuous martingale $M^c$.
If $\triangle M_t> -1, \ t>0$, then the It\^o equation

\begin{equation}\label{eq:uzj}
\mathfrak{z}_t=1+\int_0^t\mathfrak{z}_{s-}dM_s
\end{equation}
obeys the unique nonnegative solution  (eg.\cite{DoDade})
\begin{equation}\label{eq:DoLdu}
\mathfrak{z}_{t}=\exp\Big(M_t-\frac{1}{2}\langle
    M^c\rangle_t\Big)\prod_{s=0}^t(1+\triangle M_s)e^{-\triangle M_s},
\quad t\ge 0,
\end{equation}
known as Doleans-Dade exponential. It is well known that
$\mathfrak{z}$ is a nonnegative local martingale and, therefore,  it is a supermartingale with $\E \mathfrak{z}_t\le 1$ for any $t\ge 0$.  If
\begin{equation}\label{one}
\E\mathfrak{z}_{_T}=1, \ \exists \ T>0,
\end{equation}
then    $\mathfrak{z}=(\mathfrak{z}_t)_{t\in[0,T]}$ is a martingale, i.e.,
$ \E\mathfrak{z}_t\equiv 1$, $t\in[0,T]$. The property
\eqref{one} is used in different applications, where
 $\mathfrak{z}_{_T}$ plays a role of the Radon-Nikodym derivative of one probability measure w.r.t.
 another one supported on the Skorokhod space.
The random variable $\mathfrak{z}_{_T}$ is one of important objects involved in Statistic of  Random Processes (
Liptser-Shiryaev, \cite{LSI}),   in Financial Mathematics
(Shiryaev, \cite{ShAN}, Sin, Carlos A. \cite{Sin}, etc.), in the proof of existence of weak solutions of It\^o's  equations   (Rydberg \cite{Rud}) and many others important applications.

Below we give a short survey  of known conditions implying  $\E\mathfrak{z}_{_T}=1$ provided that
$M\equiv M^c$, in which case
$
\E\exp\big(M^c_T-\frac{1}{2}\langle M^c\rangle_T\big)=1.
$

Girsanov in his classical paper \cite{Girs}  used the condition
$\langle M^c\rangle_T\le \text{const}$. This condition was weakened in many variants and accomplished
by Novikov condition,\cite{Novikov},
\begin{equation*}
\E e^{\frac{1}{2}\langle M^c\rangle_T}<\infty,
\end{equation*}
by Kazamaki condition  \cite{Kaz}
\begin{equation*}
\sup\limits_{t\in[0,T]}\E e^{\frac{1}{2}M^c_t}<\infty
\end{equation*}
and, finally, by Krylov conditionÊ\cite{Kry}:
$$
    \lim\limits_{\varepsilon\downarrow 0}\varepsilon\log\E\exp\Big([1-\varepsilon]\frac{1}{2}
\langle M^c\rangle_T\Big)<\infty, \
\lim\limits_{\varepsilon\downarrow
0}\varepsilon\log\sup\limits_{t\in[0,T]}
\E\exp\Big([1-\varepsilon]\frac{1}{2}M^c_t\Big)<\infty.
$$
It would be noted that for any $\varepsilon\in (0,1/2)$ there exists a martingale $M^c$ such that

\noindent
$
\E e^{(\frac{1}{2}-\varepsilon)\langle M^c\rangle_T}<\infty,
$
while  \eqref{one} fails (see. \cite{LSI}, p. 224).

\medskip
Consider a simple example. Let  $B_t$ be Brownian motion and let
 $M^c_t =
2\int_0^t  B_s dB_s$.
\newline
Then Kazamaki's condition fails for $T\ge 1$, because
  $\E\exp\big(\frac{1}{2}B^2_{1}-\frac{1}{2}\big)=\infty$, that is,
\begin{gather*}
\E\exp\bigg(\int_0^TB_sdB_s\bigg)=\E\exp\bigg(\frac{1}{2}B^2_T-\frac{1}{2}T\bigg)=\infty.
\end{gather*}
Since Novikov's condition   implies Kazamaki's condition
  Novikov's condition fails too.

All aforementioned conditions, guaranteing  \eqref{one}, are formulated in terms of   $M^c$ and $\langle M^c\rangle$. Verifications of these conditions require often complicated
and even non-achievable   calculations.
A natural question  is  how to check \eqref{one} in the setting of the aforementioned example for $T\ge 1$.
It is surprising that it is possible to do with the help of Bene\^s's condition, \cite{Benes}.

\medskip
The main aim of this paper consists in showing that conditions referred to later as
``Bene\^s conditions''  provide
\eqref{one}  for martingales $M$ of sufficiently general structure
\eqref{eq:DoLdu}.

Together with Bene\^s' condition we shall use a uniform integrability condition, a test which was proposed by Hitsuda
in \cite{Hits}).

Note also that the  proposed approach of verification of $\E \mathfrak{z}_{_T}=1$    fits naturally with the method of establishing   $\E \mathfrak{z}_\infty=1$  in \cite{KLSh}.

Approaches used in Markov setting that do  not use Bene\^s' conditions can be found in recent paper of
Cheridito, Filipovi\'c, Yor, \cite{CFY}.

\medskip
In order to formulate our result assume $M$ in
\eqref{eq:DoLdu} is a part of some semimartingale
$X$ involved in typical models met in applications:

\begin{itemize}
\item[1)]   {\em a non-explosive Markov process}

\item[2)] {\em a non Markovian  semimartingale  }

\item[3)] {\em a possibly explosive Markov process}.
\end{itemize}
The general case is analyzed in Section \ref{Sec-3}. In Section \ref{sec-2}, we consider   the case of continuous martingales, which is the simplest   from technical point of few.
In Section \ref{sec-2} we show how to replace the classical Bene\^s proof by a new one serving most
  models.

Applicability  of the main result is shown on many examples.
An auxiliary technical result (generalized Girsanov theorem) is given in Appendix \ref{sec-A}.

\begin{remark}
The question whether  $\mathfrak{z}$  a martingale   arises    often  in view of the following problem. Let
$\mu^X$ and $\mu^Y$ be probability measures supported on Skorokhod space. These measures are distributions
of semimartingales  $(X,Y)=(X_t,Y_t)_{t\in [0,T]}$ respectively. Then the question becomes under which conditions
$\mu^X\ll \mu^Y$ with the Radon-Nikodym derivative $\frac{d\mu^Y}{d\mu^X}(X)=\mathfrak{z}_{_T}$?
These result can be found in {\rm \cite{Girs} (Girsanov)}, {\rm \cite{Daw} (Dawson)}, {\rm \cite{ItoWat}
(It\^o - Watanabe)}, {\rm \cite{KSH} (Kadota -
Shepp)}, {\rm \cite{K} (Kunita)},  {\rm \cite{LM} (L\'epingle - M\'emin)}, {\rm \cite[Ch. 7]{LSI} (Liptser - Shiryaev)}, {\rm \cite{KLSh} (Kabanov - Liptser - Shiryaev)}
and more new papers {\rm \cite{Rud} (Rydberg )}, {\rm \cite{UPR} (Palmowski - Rolski)}, {\rm \cite{WH+} (Wong - Heyde )}, {\rm \cite{CFY} (Cheridito - Filipovi\'c - Yor)}, {\rm \cite{MiUr}
 (Mijitovic - Urusov)}, {\rm \cite{BauNua} (Baudoin - Nualart)}, etc.
\end{remark}

\section{\bf  Bene\^s conditions. New   proofs}\label{sec-2}

In this Section we show   the example  where conditions of Kazamaki and Novikov fail,
 but Bene\^s' conditions don't.

We consider two types of continuous martingales:
$$
M'_t=\int_0^t\sigma(B_s)dB_s\quad\text{and}\quad M''_t=\int_0^t\sigma_s(B)dB_s,
$$
where functions $\sigma(y)$ and $\sigma_s(y)$ of arguments $y\in\mathbb{R}$ and
$(s, y_{[0,s]})\in\mathbb{R}_+\times\mathbb{C}_{[0,\infty)}$ are measurable w.r.t.
corresponding $\sigma$-algebras and satisfy the linear grows conditions:
\begin{gather}\label{BC}
\sigma^2(y)\le \texttt{r}[1+y^2],
\\
\sigma^2(s,y_{[0,s]})\le \texttt{r}\bigg[1+\sup_{s'\le s}y^2_{s'}\bigg].
\label{BCC}
\end{gather}

It is well known from the classical Bene\^s paper \cite{Benes} (see also
Karatzas, Shreve, \cite{142}), that
\eqref{one} holds   with any $T>0$. We show that this result can be easily obtained avoiding
piece-wise technique approximation, used by Bene\^s.

  Theorems \ref{theo-2.0.1} and \ref{theo-2.0.2} were formulated by Bene\^s, but the proofs below    are  new.
\begin{theorem}\label{theo-2.0.1}
Let $M_t=M'_t$ and \eqref{BC} hold true.

Then $\eqref{one}$ is valid for any $T>0$.
\end{theorem}
\begin{proof} Choose
$
\tau_n=\inf\big\{t:(\mathfrak{z}_t\vee B^2_t)\ge n\big\}\footnote{ $\inf\{\varnothing\}=\infty$},
\ n\ge 1.
$
Write
$$
\mathfrak{z}_{t\wedge\tau_n}=1+\int_0^tI_{\{s\le\tau_n\}}\mathfrak{z}_s\sigma(B_s)dB_s.
$$
The definition of $\tau_n$ and condition  \eqref{BC} imply  that the integrand
in the  It\^o's integral $\int_0^tI_{\{s\le\tau_n\}}\mathfrak{z}_s\sigma(B_s)dB_s$ is bounded. Therefore
the process  $\mathfrak{z}_{t\wedge\tau_n}$  is a square integrable martingale, that is,
$\E\mathfrak{z}_{_{T\wedge\tau_n}}=1$.
If a family $(\mathfrak{z}_{_{T\wedge\tau_n}})_{n\ge 1}$ is uniformly integrable, then
$\E\mathfrak{z}_{_{T\wedge\tau_n}}\xrightarrow[n\to\infty]{}1$ and
$
\lim_{n\to\infty}\E\mathfrak{z}_{_{T\wedge\tau_n}}=\E\mathfrak{z}_{_T}=1.
$
Thus, it is left to check   the  uniform integrability.
Following    Hitsuda we apply Vall\'ee de Poussin's theorem with function
$x\log x$, $x\ge 0$ and show that
$$
\sup_n\E \mathfrak{z}_{_{T\wedge\tau_n}}\log(\mathfrak{z}_{_{T\wedge\tau_n}})<\infty.
$$
Since $
\E\mathfrak{z}_{_{T\wedge\tau_n}}=1,
$
 change the probability measure
  $ \Q^n\ll \mathsf{P} $ with
$ d\Q^n=\mathfrak{z}_{_{T\wedge\tau_n}}d\mathsf{P} $
and obtain (here $\widetilde{\E}^n$ denotes an expectation relative to $\Q^n$)

$$
\sup_n\E \mathfrak{z}_{_{T\wedge\tau_n}}\log\big(\mathfrak{z}_{_{T\wedge\tau_n}}\big)=
\sup_n\widetilde{\E}^n\log\big(\mathfrak{z}_{_{T\wedge\tau_n}}\big).
$$
Next,
$
\log\mathfrak{z}_t=\int_0^t\sigma(B_s)dB_s-\frac{1}{2}\int_0^t\sigma^2(B_s)ds \le \int_0^t\sigma(B_s)dB_s
$
implies
$$
\sup_n\E\mathfrak{z}_{_{T\wedge\tau_n}}\log\big(\mathfrak{z}_{_{T\wedge\tau_n}}\big)\le
\sup_n\widetilde{\E}^n\int_0^{T}I_{\{s\le \tau_n\}}\sigma(B_s)dB_s
$$
which gives us a hint to use a representation of
$B_{t\wedge\tau_n}$ as a $\Q^n$ - semimartingale in the formula
$\widetilde{\E}^n\int_0^{T}I_{\{s\le \tau_n\}}\sigma(B_s)dB_s$.
By the classical Girsanov theorem
$$
B_{t\wedge\tau_n}=\int_0^tI_{\{s\le \tau_n\}}\sigma(B_s)ds+\widetilde{B}^n_t
$$
with  $\Q^n$ - Brownian motion $\widetilde{B}^n_t$ stopped at time $\tau_n$ and having the
predictable quadratic variation $\langle \widetilde{B}\rangle_t=t\wedge\tau_n$.
Hence
$$
\sup_n\E\mathfrak{z}_{_{T\wedge\tau_n}}\log\big(\mathfrak{z}_{_{T\wedge\tau_n}}\big)\le
\sup_n\widetilde{\E}^n\int_0^{T}I_{\{s\le \tau_n\}}\sigma^2(B_s)ds
$$
and the proof is reduced to verification of
$
\widetilde{\E}^n\int_0^{T}\sigma^2(B_{s\wedge\tau_n})ds\le  \texttt{r}
$
with a constant \texttt{r} independent of $n$. Bene\^s condition \eqref{BC}
is the key point of the required verification. It enables to replace
$\widetilde{\E}^n\int_0^{T}\sigma^2(B_{s\wedge\tau_n})ds$ by $\widetilde{\E}^n\int_0^{T}(B^2_{s\wedge\tau_n})ds$ and verify only the validity of
\begin{equation*}
\tilde{\E}^n\int_0^T(B^2_{s\wedge\tau_n})ds\le\texttt{r}
\end{equation*}
with \texttt{r} independent of $n$.
To this end, by applying the It\^o formula to
$B^2_{t\wedge\tau_n}$, we obtain
$$
\begin{array}{ll}
B^2_{t\wedge\tau_n}=2\int_0^tI_{\{s\le \tau_n\}}
B_{s}\sigma(B_s)ds + 2\int_0^tI_{\{s\le \tau_n\}}B_sd\widetilde{B}^n_s
+
\langle \widetilde{B}^n\rangle_t
\\ \\
\widetilde{\E}^nB^2_{t\wedge\tau_n}=2\int_0^t\widetilde{\E}^nI_{\{s\le \tau_n\}}
B_{s}\sigma(B_s)ds +\widetilde{\E}^n\langle \widetilde{B}^n\rangle_t.
\end{array}
$$
Since $|B_{s}\sigma(B_s)|ds\le c\big[1+B^2_s|$, the following upper bound holds for
$V^n_t=\widetilde{\E}^nB^2_{t\wedge\tau_n}$:
$$
\widetilde{\E}^nB^2_{t\wedge\tau_n}\le 2\int_0^t\widetilde{\E}^nI_{\{s\le \tau_n\}}
|B_{s}\sigma(B_s)|ds+\widetilde{\E}^n(t\wedge\tau_n).
$$
Therefore $V^n_t$ satisfies the Gronwall-Bellman inequality
$
V^n_t\le \texttt{r}\big[1+\int_0^tV^n_sds\big]
$
with appropriated positive constant \texttt{r}.

Hence $\int_0^tV^n_sds\le e^{\texttt{r}t}-1$, that is,
$
\sup\limits_n\widetilde{\E}^n\int_0^TB^2_{s\wedge\tau_n}ds\le e^{\texttt{r}T}-1,  \ \forall \ T>0.
$
\end{proof}

\
\bigskip

\begin{theorem}\label{theo-2.0.2}
 Let $M_t=M''_t$ and condition \eqref{BCC} holds.

Then $\E\mathfrak{z}_{_T}=1$ for any $T>0$.
\end{theorem}

\begin{proof} Choose
$
\tau_n=\inf\big\{t:(\mathfrak{z}_t\vee \sup_{s\le t}B^2_s)\ge n\big\}
$
and obtain  $\E\mathfrak{z}_{_{T\wedge\tau_n}}=1$.
Now, the proof of the theorem is reduced to verification of the inequality
$
\sup_n\widetilde{\E}^n\log\big(\mathfrak{z}_{_{T\wedge\tau_n}}\big)<\infty,
$
where
$
\widetilde{\E}^n
$
denotes the expectation   relative to the probability measure $\Q^n$, $\Q^n\ll\mathsf{P}$, $\frac{d\Q^n}{d\mathsf{P}}=\mathfrak{z}_{_{T\wedge\tau_n}}$.
By Girsanov's theorem the process  $B_{t\wedge\tau_n}$ is a semimartingale w.r.t. the measure
$\Q^n$:
\begin{equation}\label{eq:22/}
B_{t\wedge\tau_n}=\int_0^tI_{\{s\le \tau_n\}}\sigma_s(B)ds+\widetilde{B}^n_t
\end{equation}
 with $\Q^n$ñ Brownian motion $\widetilde{B}^n_t$ stopped at time $\tau_n$, having the predictable
 variation $\langle \widetilde{B}\rangle_t=t\wedge\tau_n$.
Consequently
$
\widetilde{\E}^n\int_0^tI_{\{s\le \tau_n\}}\sigma_s(B)dB_s=\widetilde{\E}^n\int_0^tI_{\{s\le \tau_n\}}\sigma^2_s(B)ds
$
and, in view of condition \eqref{BCC},
$$
\widetilde{\E}^n\int_0^tI_{\{s\le \tau_n\}}\sigma^2_s(B)ds\le \texttt{r}\big[1+\widetilde{\E}^n\int_0^tI_{\{s\le \tau_n\}}\sup_{s'\le s}B^2_{s'}ds\big].
$$
So, it suffices to prove that
\begin{equation}\label{eq:sp}
\sup\limits_n\widetilde{\E}^n\int_0^T\sup_{s'\le s}B^2_{s\wedge\tau_n}ds<\infty.
\end{equation}
Condition \eqref{eq:22/} and Cauchy-Schwartz inequality enable us to use the following inequality
\begin{gather*}
\sup_{t'\le t} B^2_{t'\wedge\tau_n}\le 2\bigg|\int_0^tI_{\{s\le \tau_n\}}|\sigma_s(B)|ds\bigg|^2+2\sup_{t'\le t}|\widetilde{B}^n_{t'\wedge\tau_n}|^2
\\
\le 2t\int_0^tI_{\{s\le \tau_n\}}\sigma^2_s(B)ds+2\sup_{t'\le t}|\widetilde{B}^n_{t'\wedge\tau_n}|^2
\\
\le 2t\int_0^tI_{\{s\le \tau_n\}}\texttt{r}\bigg[1+\sup_{s'\le s}\sigma^2_{s'}(B)\bigg]ds+
2\sup_{t'\le t}|\widetilde{B}^n_{t'\wedge\tau_n}|^2.
\end{gather*}
Moreover, the Doob maximal inequality  for square integrable martingales yields\\
$\widetilde{\E}^n
\sup_{t'\le t}|\widetilde{B}^n_{t'\wedge\tau_n}|^2\le 4(t\wedge\tau_n)$.
These inequalities allow  us to use  Gronwall-Bellman inequality
$
V^n_t\le \texttt{r}\big[1+\int_0^tV^n_sds\big], \ t\in[0,T],
$
where
$V^n_t:=\widetilde{\E}^n\sup_{t'\le t}B^2_{t\wedge\tau_n}$,
and obtain \eqref{eq:sp}.
\end{proof}

\section{\bf General model }\label{Sec-3}

Assume a martingale $M=M^c+M^d$  is a part of some semimartingale $X$.
While it may not be `the most general case',   it covers, however, most examples met in applications.
Other generalizations in the spirit of our arguments are possible but we decided not pursue them here, as the reader will see that the paper is already technical enough.

Introduce the following notations and assumptions.
\begin{itemize}
  \item $(\varOmega, F, (\mathcal{F}_t)_{t\in[0,T]}, \mathsf{P})$ is a stochastic basis satisfying the
general conditions.
  \item $\mathbb{C}=\mathbb{C}_{[0,T]}$ è $\mathbb{D}=\mathbb{D}_{[0,T]}$ are the spaces of continuous
functions and the Skorokhod space of c\'adl\'ag functions.

  \item
      $x_{[0,t)}=\{x_{t'}:t'<t\}$ and $x_{[0,t]}=\{x_{t'}:t'\le t\}$.

\item $(B_t)_{t\in[0,T]}$ is standard  Brownian motion.

\item $\mu(dt,dz)$ is integer  valued random measure on
$[0,T]\times\mathbb{R}_+$ (see \cite{JSh} and \cite{LSMar}).

\item $\nu(dt,dz):=dtK(dz)$ is a compensator  (Lev\'y's measure) of $\mu(dt,dz)$,
here $K(dz)$ is $\sigma$ -finite measure supported on $\mathbb{R}$ such that
\begin{equation*}
\int\limits_\mathbb{R}z^2K(dz)<\infty.
\end{equation*}

\item $a_s(x)$, $b_s(x)$, $\sigma_s(x)$ and $h_s(x,z)$, $\varphi_s(x,z)$ are measurable
(with respect to appropriate $\sigma$-algebra) functions
of arguments $(s,x_{[0,s)})$ and $(s,x_{[0,s)}),z)$, where $s\in[0,T]$, $x_{[0,s)}\in
\mathbb{D}_{[0,T]}$, $z\in\mathbb{R}$; for any $x\in \mathbb{D}_{[0,T]}$ functions
$a_s(x)$, $b_s(x)$, $\sigma_s(x)$ are square integrable with respect to $ds$ on $[0,T]$
as well as for any $x\in\mathbb{D}_{[0,T]}$ and any $z\in \mathbb{R}$ functions
$h_s(x,z)$ and $\varphi_s(x,z)$ are square integrable with respect to $dsK(dz)$
on $[0,T]\times \mathbb{R}$; \footnote{in the case of explosion these properties are assumed to be valid
on any open time interval up to the  moment of explosion.}

 Furthermore, we assume that
\begin{equation*}
\varphi_s(x,z)>-1
\end{equation*}
which is equivalent to $\Delta M_s>-1$.
\item \texttt{r} is a generic positive constant taking different values in different
places
\item $\inf\{\varnothing\}=\infty$.
\end{itemize}

\medskip
\noindent
The semimartingale  $X$ is assumed to be a unique weak solution of the It\^o equation:
\begin{multline}\label{eq:sde}
X_t=X_0+\int_0^ta_s(X)ds+\int_0^tb_s(X)dB_s
+\int_0^t\int_{\mathbb{R}}h_s(X,z)[\mu(ds,dz)-dsK(dz)].
\end{multline}
The martingale  $M=M^c+M^d$ driven by $X$  is given by
\begin{equation}\label{eq:Mc}
 M_t=\underbrace{\int_0^t\sigma_s(X)dB_s}_{=M^c_t}
+\underbrace{\int_0^t\int_\mathbb{R}\varphi_s(X,z)[\mu(ds,dz)-dsK(dz)]}_{=M^d_t}.
\end{equation}

\section{$\pmb{X}$ is a non-explosive Markov process}\label{sec-3.1}

\subsection{Main result. Examples}
\label{sec-3.1.1}

Since
$X$ is a Markov process, we refine its description, that is, we replace \eqref{eq:sde} by the It\^o equation:
\begin{align*}
X_t=X_0+\int_0^ta_s(X_{s-})ds+\int_0^tb_s(X_{s-})dB_s
+\int_0^t\int_{\mathbb{R}}h_s(X_{s-},z)[\mu(ds,dz)-dsK(dz)],
\end{align*}
where $a_s(X_{s-}):=a(s,X_{s-})$, $b_s(X_{s-})=b(s,X_{s-})$ è $h_s(X_{s-},z)=h(s,X_{s-},z)$.

We use the same notations  in \eqref{eq:uzj}:
\begin{multline*}
\mathfrak{z}_t=1+\int_0^t\mathfrak{z}_{s-}dM_s
\\
=1+\int_0^t\mathfrak{z}_{s-}\Big[\sigma_s(X_{s-})dB_s
+\int_0^t\int_\mathbb{R}\varphi_s(X_{s-},z)[\mu(ds,dz)-dsK(dz)\Big].
\end{multline*}

Introduce the following operators  (depending on $s$), acting on $ (x_s)_{s\in[0,T]}\in \mathbb{D}$
\begin{equation*}
L_s(x_{s-}):=2x_{s-} a_s(x_{s-})+b^2_s(x_{s-})+\int_\mathbb{R}h^2_s(x_{s-},z);
\end{equation*}
\begin{multline}\label{eq:magl+}
\mathfrak{L}_s(x_{s-}):=2x_{s-}\Big[a_s(x_{s-})+b_s(x_{s-})\sigma_s(x_{s-})
+\int_\mathbb{R}h_s(x_{s-},z)\varphi_s(x_{s-},z)K(dz)\Big]
\\
+b^2_s(x_{s-})+\int_\mathbb{R}h^2_s(x_{s-},z)K(dz)+\int_\mathbb{R}
h^2_s(x_{s-},z)\varphi_s(x_{s-},z)K(dz).
\end{multline}

The role of the operator  $L_s(x_{s-})$ is clarified as follows.
\begin{proposition}\label{lem-none}
Let $X^2_0\le \texttt{r}$ and $L_s(x_{s-}) \le \texttt{r}[1+x^2_{s-}]$.
Then \begin{equation*}
\sup_{t\in[0,T]}\E X^2_t<\infty,
\end{equation*}
that is, the process $X$ does not explode  on any finite time interval $[0,T]$.
\end{proposition}
\begin{proof}
Applying the It\^o formula to  $X^2_t$ we obtain
$
X^2_t=X^2_0+\int_0^tL_s(X_{s-})ds+\mathcal{M}_t,
$
where $\mathcal{M}_t$ is a local martingale.

Note also that the definition of
$L_s(x_{s-})$
and
$\zeta_n=\inf\{t:X^2_t\ge n\}$ provide  $\E X^2_{t\wedge\zeta_n}<\infty$. So,
$
X^2_{t\wedge\zeta_n}=X^2_0+\int_0^tI_{\{s\le\zeta_n\}}L_s(X_{s-})ds +\mathcal{M}_{t\wedge\zeta_n},
$
implies
$\E\mathcal{M}^2_{t\wedge\tau_n}<\infty$ and
$\E \mathcal{M}_{t\wedge\zeta_n}=0$.
Also  $V^n_t:=\E X^2_{t\wedge\zeta_n}$ solves Gronwall-Bellman's inequality
$$
\E V^n_n\le \texttt{r}\big[1+\int_0^tV^n_sds\big].
$$
Hence
$
\E X^2_{t\wedge\zeta_n}\le \texttt{r}e^{\texttt{r}t}
$
and by Fatou theorem
$
\E X^2_{t}\le \texttt{r}e^{\texttt{r}t}.
$

So, $\sup_{t\in[0,T]}\E X^2_t\le \texttt{r}e^{\texttt{r}T}$.
\end{proof}

\begin{remark}\label{rem-3a}
Looking ahead let us clarify a role of the operator $\mathfrak{L}_s(x_{s-})$. Assume
$\E\mathfrak{z}_{_T}=1$. So $\Q\ll\mathsf{P}$ with $\frac{d\mathsf{P}}{d\mathsf{P}}=\mathfrak{z}_{_T}$
is the probability measure. Then the operator  $\mathfrak{L}_s(x_{s-})$ plays the role of
the operator
$L_s(x_{s-})$ for the process $X_t$ under the new measure  $\Q$.
\end{remark}

\begin{theorem}\label{theo-4.1.2}
Ïóñòü $|X_0|\le \texttt{r}$ è

{\rm 1)} $\sigma^2_s(x_{s-})+\int_\mathbb{R}\varphi^2_s(x_{s-},z)K(dz)\le \texttt{r}\Big[1+x^2_{s-}\Big]$

{\rm 2)} $L_s(x_{s-})\le \texttt{r}[1+x^2_{s-}]$

{\rm 3)} $\mathfrak{L}_s(x_{s-})\le \texttt{r}[1+x^2_{s-}]$

\noindent
Then $ (\mathfrak{z}_t)_{t\in[0,T]}$ is the martingale for any $T>0$.
\end{theorem}

\mbox{}

\noindent
The proof of this theorem is given in Section \ref{sec-4.2}.
Now, we  illustrate its applications in various models.

To avoid repitions we omit  ``... the conditions of the theorem  are fulfilled ...'' and ``...for
any $T>0,..$'' and  use  a shorthand ``$\mathfrak{z}=\mathcal{E}(M)$''  instead, wherever it does not cause misunderstanding.

\begin{example}[\cite{BeLip}] {\rm $X$ is purely discontinuous martingale with independent increments. Let
$$
X_t=\int_0^t\int_{\mathbb{R}}z[\mu(ds,dz)-K(dz)ds] \ \text{and} \
M_t=\int_0^t\int_{\mathbb{R}}|X_{s-}z|[\mu(ds,dz)-K(dz)ds]
$$
Since
\begin{itemize}
  \item $X_0=0$,
  \item $a_s(x_{s-})=b_s(x_{s-})=0$ è $h_s(x_{s-},z)=z$,
  \item $\sigma_s(x_{s-})=0$ è $\varphi_s(x_{s-},z)=|x_{s-}z|$,
  \item $\int\limits_\mathbb{R}z^2K(dz)<\infty$,
\end{itemize}
then together with additional condition  $\int\limits_\mathbb{R}|z|^3K(dz)<\infty$ we have
$$
\left.
  \begin{array}{ll}
   \sigma^2_s(x_{s-})+ \int_\mathbb{R}\varphi^2_s(x_{s-},z)K(dz)= x^2_{s-}\int_\mathbb{R}z^2K(dz)&  \\
    L_s(x_{s-})= \int_\mathbb{R}z^2K(dz)&  \\
    \mathfrak{L}_s(x_{s-})=2x_{s-}|x_{s-}| \int_\mathbb{R}z|z|K(dz)+\int_\mathbb{R}z^2K(dz)+|x_{s-}|\int_\mathbb{R}|z|^3K(dz)&
  \end{array}
\right\}
$$
$$
\le \texttt{r}[1+x^2_{s-}].
$$

Thus, $\mathfrak{z}=\mathcal{E}(M)$.}
\end{example}

\begin{example}\label{ex-3.1.2}
{\rm Constant Elasticity of Variance, {\rm \cite {AP},
\cite{Cox}, \cite{DelShir}}.

Let
$$
X_t=1+\int_0^tX_sds+\int_0^t\sqrt{X^+_s}dB_s\quad\text{and}\quad
M_t=\int_0^t\sqrt{X^+_s}dB_s.
$$
Then,
\begin{itemize}
  \item $X_0=1$
  \item $a_s(x_{s-})=x_{s-}$,  $b_s(x_{s-})=\sqrt{x^+_{s-}}$ and $h_s(x_{s-},z)=0$,
  \item $\sigma_s(x_{s-})=\sqrt{x^+_{s-}}$ and $\varphi_s(x_{s-},z)=0$.
\end{itemize}
Consequently
$$
\left.
  \begin{array}{ll}
   \sigma^2_s(x_{s-})+ \int_\mathbb{R}\varphi^2_s(x_{s-},z)K(dz)=x^+_{s-} &  \\
    L_s(x_{s-})=2x^2_{s-} +x^+_{s-}&  \\
    \mathfrak{L}_s(x_{s-})=2x_{s-}[x_{s-}+x_{s-}^+]+x_{s-}^+ &
  \end{array}
\right\}\le \texttt{r}[1+x^2_{s-}]
$$
It should be noted that $\vartheta=\inf\{t:X_t=0\}$ is the time to
absorbtion at zero of the process $X_t$ (time to ruin), $\vartheta<\infty$ with a positive
probability, and $X_t=X_{t\wedge \vartheta}$.
Thus $(\mathfrak{z}_{_{t\wedge\vartheta}})_{t\le T}$ is a martingale for any $T>0$.}
\end{example}

\begin{example}\label{ex-3.1.3}
{\rm Cubic stabilizing drift. Let
$$
X_t=1-\int_0^tX^3_{s}ds+\int_0^tX_{s}dB \quad\text{and}\quad M_t=\int_0^tX_{s}dB_s.
$$
Then,
\begin{itemize}
\item $X_0=1$,
\item $a_s(x_{s-})=-x^3_{s-}$,  $b_s(x_{s-})=x_{s-}$ è $h_s(x_{s-},z)=0$,
  \item $\sigma_s(x_{s-})=x^+_{s-}$ è $\varphi_s(x_{s-},z)=0$.
\end{itemize}
So,
$$
\left.
  \begin{array}{ll}
   \sigma^2_s(x_{s-})=(x^+_{s})^2&  \\
    L_s(x_{s-})=-2x^4_s+x^2_s &  \\
    \mathfrak{L}_s(x_{s-})=-2x^4_s+x^2_s+2|x^3_s| &
  \end{array}
\right\}\le \texttt{r}[1+x^2_{s-}].
$$
Thus, $\mathfrak{z}={\mathcal E}(M)$.}
\end{example}

\begin{example}
\label{ex-3.1.4}{\rm Brownian Bridge. Zero mean
Gaussian process $X_t$ defined on the time interval is said Brownian bridge
if its correlation function $$
R(t',t)=(t'\wedge t)[1-(t'\vee t)].
$$
It is also  well known that $X_t$ is the unique solution of It\^o's equation
$$
X_t=-\int_0^t\frac{X_s}{1-s}ds+B_t, \quad t\in [0,1), \quad \lim_{t\uparrow 1}X_t=0.
$$
Let $M_t=\int_0^tX_sdB_s$,  $t\le 1$.

Here $T=1$ and
\begin{itemize}
\item $X_0=0$,
\item $a_s(x_{s-})=-\frac{x_{s-}}{1-s} \ (s<1)$,  $b_s(x_{s-})=x_{s-}$ è $h_s(x_{s-},z)=0$,
  \item $\sigma_s(x_{s-})=x_{s-}$ è $\varphi_s(x_{s-},z)=0$.
\end{itemize}
Therefore
$$
\left.
  \begin{array}{ll}
   \sigma^2_s(x_{s-})=x^2_{s-}&  \\
    L_s(x_{s-})=-2x^4_{s-}+x^2_{s-} &  \\
    \mathfrak{L}_s(x_{s-})=-2x^4_{s-}+x^2_s+2|x^3_{s-}| &
  \end{array}
\right\}\le \texttt{r}[1+x^2_{s-}]
$$
and by Theorem \ref{theo-4.1.2}, $\E\mathfrak{z}_1=1.$}
\end{example}

\begin{example} \label{ex-333} {\rm One extension of Mijitovic and Urusov
example (see \cite{MiUr})

Let $\alpha\in (-1,0]$ and
\begin{align*}
X_t=1+\int_0^t|X_{s}|^\alpha ds+B_t\quad\text{and}\quad M_t=\int_0^tX_{s}dB_s.
\end{align*}
In {\rm \cite{MiUr}}, it is shown  ïîêàçàíî, ÷òî
 $\mathfrak{z}={\mathcal E}(M)$.
Theorem \ref{theo-4.1.2} enables to show that $\mathfrak{z}={\mathcal E}_T(M)$ even if
$\alpha=-1$. In this case $X_t$ is the Bessel process {\rm (see e.g. Exercise 2.25 p. 197 \cite{RevYor})}, that is,
\begin{equation*}
X_t=1+\int_0^t\frac{ds}{X_s}+B_t.
\end{equation*}
Then
\begin{itemize}
\item $X_0=1$,
\item $a_s(x_{s-})=\frac{1}{x_{s-}\vee 0}$,  $b_s(x_{s-})=1$ è $h_s(x_{s-},z)=0$,
  \item $\sigma_s(x_{s-})=x_{s-}$ è $\varphi_s(x_{s-},z)=0$
\end{itemize}
and
$$
\left.
  \begin{array}{ll}
   \sigma^2_s(x_{s-})=x^2_{s-}&  \\
    L_s(x_{s-})= 3&  \\
    \mathfrak{L}_s(x_{s-})=3+2x_{s-} &
  \end{array}
\right\}\le \texttt{r}[1+x^2_{s-}].
$$
Therefore $\mathfrak{z}={\mathcal E}(M)$.}
\end{example}

\section{$\pmb{X} $ \bf is a past-dependent semimartingale}\label{sec-6}
\subsection{Main result. Examples}
Recall representations \eqref{eq:sde} and \eqref{eq:Mc}. Operators
$L_s(x)$ and  $\mathfrak{L}_s(x)$ are changed as follows
\begin{equation*}
L_s(x):=a^2_s(x)+b^2_s(x)+\int_\mathbb{R}h^2_s(x,z)K(dz),
\end{equation*}
\begin{equation}\label{eq:LLL}
\begin{array}{ll}
\hskip 1.3in\mathfrak{L}_s(x):=a^2_s(x)+b^2_s(x)+\int_\mathbb{R}h^2_s(x,z)K(dz)
+b^2_s(x)\sigma^2_s(x)
\\ \\
\hskip 1in+\int_\mathbb{R}h^2_s(x,z)K(dz)\int_\mathbb{R}\varphi^2_s(x,z)K(dz)
+\int_\mathbb{R}h^2_s(x,z)\varphi_s(x,z)K(dz).
\end{array}
\end{equation}

\begin{theorem}\label{theo-5.0.1}
If $|X_0|\le \texttt{r}$ è

{\rm 1)} $\sigma^2_s(x)+\int_\mathbb{R}\varphi^2_s(x,z)K(dz)\le
  \texttt{r}\bigg[1+\sup\limits_{s'<s}x^2_{s'}\bigg]$

{\rm 2)} $L_s(x)\le \texttt{r}\bigg[1+\sup\limits_{s'<s}x^2_{s'}\bigg]$

{\rm 3)} $\mathfrak{L}_s(x )\le \texttt{r}\bigg[1+\sup\limits_{s'<s}x^2_{s'}\bigg]$,

\noindent
then $\mathfrak{z}=\mathcal{E}(M)$.
\end{theorem}

Note that the process $X$ does not explode due to assumption  2).

\noindent
The proof of this theorem is given in Section \ref{sec-6.3}.

\begin{example} {\rm Weak existence  for a past-dependent SDE with unit diffusion.

We show that a stochastic differential equation
\begin{equation*}
X_t=\int_0^ta_s(X)ds+B_t.
\end{equation*}
has a weak solution on any time interval $[0,T]$ if
\begin{equation*}
a^2_s(y)\le \texttt{r}\big[1+\sup_{s'\le  s}y^2_{s}\big], \ (y_s)_{s\in[0,T]}\in\mathbb{C}
\end{equation*}
{\rm(cf   assumptions of Theorem 7.2, Ch. 7, \S 7.2 â \cite{LSI})}.

Set  $M_t= \int_0^t\sigma_s(B)dB_t$ with
$\sigma_s(y)\equiv a_s(y)$. Then by Theorem \ref{theo-5.0.1},
$\E\mathfrak{z}_{_T}=1$. So, there exists a probability measure $\Q\ll\mathsf{P}$,  $\frac{d\Q}{d\mathsf{P}}=\mathfrak{z}_{_T}$. Hence, by Girsanov theorem,
the process $(B_t,\Q)_{t\in[0,T]}$ is nothing but a weak solution of It\^o's equation
$
B_t=\int_0^ta_s(B)ds+\widetilde{B}_t
$
with $\Q$-Brownian motion $\widetilde{B}_t$.

{\rm Note that weak uniqueness of this equation also holds  (see Theorem 4.12 in \cite[Ch. 4]{LSI})}}.
\end{example}

\begin{example} {\rm A past-dependent SDE with a singular diffusion.
Assume the It\^o equation
$
X_t=X_0+\int_0^tb_s(X)dB_s
$
with $b^2_s(x)\ge 0$ obeys a weak solution. Then  equation with drift 
\begin{equation*}\label{eg:X=SB}
X_t=X_0+\int_0^ta_s(X)ds+\int_0^tb_s(X)dB_s.
\end{equation*}
also has a weak solution.
Suppose
\begin{itemize}
  \item $X^2_0\le \texttt{r}$
  \item $a_s(x)=a_s(x)I_{\{b^2(x)>0\}}$
  \item
$
\left.
          \begin{array}{ll}
            a^2_s(x)&  \\
            b^2_s(x) & \\
            \frac{a^2_s(x)}{b^2_s(x)}I_{\{b^2_s(x)>0\}} &
          \end{array}
        \right\}\le \texttt{r}\Big[1+\sup\limits_{s'\le s}x^2_{s'}\Big], \ s\in[0,T], \ (x_s)_{s\in[0,T]}
\in \mathbb{C}.
$
\end{itemize}
{\rm(cf Ch. 7, \S 7.6, Theorem 7.19 in \cite{LSI})}.
To this end, we choose
$$
\sigma_s(x)= \frac{a_s(x)}{b_s(x)}I_{\{b^2_s(x)>0\}}\quad\text{and}\quad
M_t=\int_0^t \sigma_s(B)dB_s
$$
and then apply Theorem \ref{theo-5.0.1} with
$
\mathfrak{z}_t=1+\int_0^t\mathfrak{z}_s\sigma_s(x)dB_s.
$
Since  $\E\mathfrak{z}_{_T}=1$ there exists the probability measure
$\Q\ll\mathsf{P}$ with density
$\frac{d\Q}{d\mathsf{P}}=\mathfrak{z}_{_T}$.

Then, by Girsanov theorem
$$
X_t=X_0+\int_0^t\underbrace{b_s(X)\sigma_s(X)}_{=a_s(X)}ds+\int_0^tb_s(X)d\widetilde{B}_s,
$$
with a $\Q$-Brownian motion  $\widetilde{B}_t$}.
\end{example}

\begin{example}{\rm SDE with delay.
Theorem \ref{theo-5.0.1} is applicable to semimartingale with past-dependent characteristic
in a form of delay.  This stochastic model is used often in modern stochastic control.

Let  $\vartheta>0$ denote fixed delay parameter in It\^o's equation
 $$
 X_t =I_{\{X_u\in[-\theta, 0]\}}+\int_0^ta_s(X_{s-\vartheta})ds+\int_0^tb_s(X_{s-\vartheta})dB_s.
$$
Let $M_t=\int_0^t \sigma_s(X_{s-\vartheta})dB_s$ and so
$
\mathfrak{z}_t=1+\int_0^t\mathfrak{z}_s\sigma_s(X_{s-\vartheta})dB_s.
$

Then, by Theorem \ref{theo-5.0.1}, $\mathfrak{z}={\mathcal E}(M)$ if the following conditions are satisfied: $$
\left.
\begin{array}{ll}
a^2_s(x_{s-\vartheta})+b^2_s(x_{s-\vartheta})&
\\
\sigma^2_s(x_{s-\vartheta}) &
\\
 \sigma^2_s(x_{s-\vartheta})b^2_s(x_{s-\vartheta})
&
  \end{array}
\right\}\le \texttt{r}\Big[1+\sup\limits_{s'\le s}x^2_{s'}\Big], \ s\in[0,T], \ (x_s)_{s\in[0,T]}
\in \mathbb{C}.
$$
}
\end{example}

\section{\bf Proofs}\label{secProofs}

\subsection{Auxiliary result}

Proofs of Theorem    \ref{theo-2.0.1} and  Theorem \ref{theo-5.0.1} follow the same idea, and rely on an auxiliary result that allows to check uniform integrability in terms of
$$
\sup_n\E\mathfrak{z}_{_{T\wedge\tau_n}}\log(\mathfrak{z}_{_{T\wedge\tau_n}})<\infty.
$$
Let $X$ and $M$ be defined by \eqref{eq:sde} and \eqref{eq:Mc}, and
$$
\mathfrak{z}_t=1+\int_0^t\mathfrak{z}_{s-}\sigma_s(X)dB_s
+\int_0^t\int_\mathbb{R}\mathfrak{z}_{s-}
\varphi_s(X,z)[\mu(ds,dz)-dsK(dz)].
$$
 Set
the localizing sequences:
\begin{equation}\label{eq:taun}
\begin{array}{cc}
\tau_n=\inf\{t:(\mathfrak{z}_t\vee X^2_t)\ge n \} \quad\text{in the proof of Theorem \ref{theo-2.0.1}}
\\
and
\\
\tau_n=\inf\{t:(\mathfrak{z}_t\vee \sup_{s\le t}X^2_s)\ge n \}\quad\text{ in the proof of Theorem \ref{theo-5.0.1}}
\end{array}
\end{equation}
and notice that $\mathfrak{z}_{(s\wedge\tau_n)-}$ and $X^2_{(s\wedge\tau_n)-}$ are bounded processes.
Since Doleans-Dade's formula \eqref{eq:DoLdu}, with martingales $M^c_t$ and $M^d_t$ defined
in \eqref{eq:Mc}, is the unique solution of SDE
$$
\mathfrak{z}_t=1+\int_0^t\mathfrak{z}_{s-}\sigma_s(X)dB_s
+\int_0^t\int_\mathbb{R}\mathfrak{z}_{s-}
\varphi_s(X,z)[\mu(ds,dz)-dsK(dz)],
$$
we have that
\begin{align}\label{eq:34u}
\mathfrak{z}_{_{t\wedge\tau_n}}&=1+\int_0^tI_{\{s\le\tau_n\}}\mathfrak{z}_{_{(s\wedge\tau_n)-}}
\sigma_s(X)
dB_s
\nonumber\\
&\quad
+\int_0^t\int_\mathbb{R}I_{\{s\le\tau_n\}}\mathfrak{z}_{_{(s\wedge\tau_n)-}}
\varphi_s(X,z)[\mu(ds,dz)-dsK(dz)].
\end{align}
Hence
\begin{equation*}\label{zBound}
\E\big(\mathfrak{z}_{_{T\wedge\tau_n}}-1\big)^2=\E\int_0^TI_{\{s\le\tau_n\}}\mathfrak{z}^2_{_{(s\wedge\tau_n)-}}
\Big(\sigma^2_s(X)+\int_\mathbb{R}\varphi^2_s(X,z)K(dz)\Big)ds.
\end{equation*}
By one of the (BC) conditions:
for any $\ s\in[0,T], \ (x_s)_{s\le T}\in\mathbb{D}$
\begin{align}
\label{eq:B-} \sigma^2_s(x)+\int_\mathbb{R}\varphi^2_s(x,z)K(dz)\le \texttt{r}\left\{
  \begin{array}{ll}
 \Big[1+x^2_{s-}\Big]& \text{ if $X$ - (BC-Markov)}\\
\\
 \Big[1+\sup_{s'< s}x^2_{s'}\Big]& \text{ if $X$ - (BC-Past Dependent)}
  \end{array}
\right.
\end{align}
It now follows from  \eqref{zBound} that $\mathfrak{z}_{_{t\wedge\tau_n}}$ is a square integrable martingale with $\E\mathfrak{z}_{_{t\wedge\tau_n}}=1$.

\begin{lemma}\label{lem-unint}
The  family $\{\mathfrak{z}_{_{T\wedge\tau_n}}\}_{n\ge 1}$
is uniformly integrable if
\begin{equation*}
\left.
\begin{array}{ll}
\sup_n\widetilde{\E}^n\int_0^TX^2_{s\wedge\tau_n}ds, & \text{in Theorem \ref{theo-4.1.2}}
\\ \\
\sup_n\widetilde{\E}^n\int_0^T\sup_{s'< s}X^2_{s'\wedge\tau_n}ds, & \text{in Theorem \ref{theo-5.0.1}}
  \end{array}
\right\}<\infty.
\end{equation*}
\end{lemma}
\begin{proof}
The existence of probability measure $\Q^n\ll \mathsf{P}$ with the density
$
\frac{d\Q^n}{d\mathsf{P}}=\mathfrak{z}_{_{T\wedge\tau_n}}
$
is obvious. Henceforth  $\widetilde{\E}^n$ is the expectation of  $\Q^n$ measure.
The uniform integrability of the family  $\{\mathfrak{z}_{_{T\wedge\tau_n}}\}_{n\to\infty}$
is verified by the Vall\'ee de Poussin theorem  with the function $x\log(x), x\ge 0$.
The formula
$
\sup\limits_n\E\mathfrak{z}_{_{T\wedge\tau_n}}\log\big(\mathfrak{z}_{_{T\wedge\tau_n}}\big)<\infty
$
is convenient since
$$
\mathfrak{z}_{_{t\wedge\tau_n}}=\exp\big(M_{t\wedge\tau_n}-A_{t\wedge\tau_n}\big),
$$
where $M_{t\wedge\tau_n}$ is a square integrable martingale and  $A_{t\wedge\tau_n}$ is an increasing positive process:
\begin{equation*}
\begin{array}{ll}
 M_{t\wedge\tau_n}=
\int_0^tI_{\{s\le\tau_n\}}\sigma_s(X)dB_s
\\
\hskip .5in+
\int_0^t\int_\mathbb{R}I_{\{s\le\tau_n\}}\varphi_s(X,z)[\mu(ds,ds)-dsK(dz)]
\\
\\
A_{t\wedge\tau_n}=\frac{1}{2}\int_0^tI_{\{s\le\tau_n\}}
\sigma^2_s(X)ds
\\
\hskip .5in+
\int_0^T\int_\mathbb{R}I_{\{s\le\tau_n\}}\big\{\varphi_s(X,z)-\log\big[1+
\varphi_s(X,z)\big]\big\}\mu(ds,dz).
\end{array}
\end{equation*}
The condition
$
\varphi_s(X,z)>-1
$
implies
$
\varphi_s(X,z)-\log\big[1+\varphi_s(X,z)\big] \ge 0.
$
This inequality, jointly with $\sigma^2_s(X) \ge 0 $, implies
$
A_{t\wedge\tau_n}\ge 0.
$
Therefore
$
\log\big(\mathfrak{z}_{_{T\wedge\tau_n}}\big)\le M_{T\wedge\tau_n}
$
and, so,
\begin{equation*}
\mathfrak{z}_{_{T\wedge\tau_n}}\log\big(\mathfrak{z}_{_{T\wedge\tau_n}}\big)\le
\mathfrak{z}_{_{T\wedge\tau_n}}M_{T\wedge\tau_n}.
\end{equation*}
Both processes $\mathfrak{z}_{_{t\wedge\tau_n}}$ and $M_{t\wedge\tau_n}$
are square integrable martingales heaving continuous and purely discontinuous components:
$
\mathfrak{z}_{_{t\wedge\tau_n}}=\mathfrak{z}^c_{_{t\wedge\tau_n}}+\mathfrak{z}^d_{_{t\wedge\tau_n}},
$
$
M_{t\wedge\tau_n}=M^c_{t\wedge\tau_n}+M^d_{t\wedge\tau_n},
$
where
\begin{equation*}
\begin{array}{ll}
 \mathfrak{z}^c_{_{T\wedge\tau_n}}=\int_0^TI_{\{s\le\tau_n\}}\mathfrak{z}_{s-}\sigma_s(X)dB_s,
\\
 M^c_{T\wedge\tau_n}=\int_0^TI_{\{s\le\tau_n\}}\sigma_s(X)dB_s,
 \\
 \\
 \mathfrak{z}^d_{_{T\wedge\tau_n}}=1+\int_0^T\int_\mathbb{R}I_{\{s\le\tau_n\}}\mathfrak{z}_{s-}
 \varphi_s(X,z)[\mu(ds,dz)-dsK(dz)],
 \\
 M^d_{T\wedge\tau_n}=\int_0^T\int_\mathbb{R}I_{\{s\le\tau_n\}}\varphi_s(X,z)[\mu(ds,ds)-dsK(dz)].
\end{array}
\end{equation*}
Hence $\E M_{T\wedge\tau_n}\mathfrak{z}_{_{T\wedge\tau_n}}=\widetilde{\E}^nM_{T\wedge\tau_n} $.
Also
\begin{gather*}
\E M^{c,n}_{T\wedge\tau_n}\mathfrak{z}_{_{T\wedge\tau_n}}=\E\int_0^TI_{\{s\le\tau_n\}}\mathfrak{z}_s
\sigma^2_s(X)ds=\E\mathfrak{z}_{_{T\wedge\tau_n}}\int_0^TI_{\{s\le\tau_n\}}
\sigma^2_s(X)ds
\\
=\widetilde{\E}^n\int_0^TI_{\{s\le\tau_n\}}
\sigma^2_s(X)ds
\end{gather*}
and
\begin{gather*}
\E M^{d,n}_{T\wedge\tau_n}\mathfrak{z}_{_{T\wedge\tau_n}}=\E\int_0^T\int_{\mathbb{R}}I_{\{s\le\tau_n\}}
\mathfrak{z}_{s-}\varphi^2_s(X,z)K(dz)ds
\\
=\widetilde{\E}^n\int_0^T\int_\mathbb{R}I_{\{s\le\tau_n\}}
\varphi^2_s(X)K(dz)ds.
\end{gather*}
So
\begin{gather*}
\widetilde{\E}^nM_{T\wedge\tau_n}
=\widetilde{\E}^n\int_0^TI_{\{s\le\tau_n\}}\bigg(\sigma^2_s(X)+\int_\mathbb{R}
\varphi^2_s(X,z)K(dz)\bigg)ds.
\end{gather*}
Now, \eqref{eq:B-} enables to finish the proof:
\begin{gather*}
\sup_n \E\mathfrak{z}_{_{T\wedge\tau_n}}\log\big(\mathfrak{z}_{_{T\wedge\tau_n}}\big)
 \le \texttt{r}
  \begin{cases}
   T+\sup\limits_{n}\widetilde{\E}^n\int_0^TX^2_{s\wedge\tau_n}ds  , & \text{for theorem \ref{theo-4.1.2}} \\ \\
   T+\sup\limits_{n}\widetilde{\E}^n\int_0^T\sup_{s'\le s}X^2_{s'\wedge\tau_n}ds   & \text{for theorem  \ref{theo-5.0.1}}
  \end{cases}
\end{gather*}
if conditions of the lemma fulfilled.
\end{proof}

\subsection{Proof of Theorem \ref{theo-4.1.2}}
\label{sec-4.2}
By Lemma \ref{lem-unint}, it suffices to verify that
$$
\sup_n\widetilde{\E}^n\int_0^TX^2_{s\wedge\tau_n}ds<\infty
$$
 with
  $\tau_n$ defined in \eqref{eq:taun} and $\widetilde{\E}^n$  the expectation
under $\Q^n$: $d\Q^n=\mathfrak{z}_{_{T\wedge\tau_n}}d\mathsf{P}$.
Thus we need to know how $X$ looks like under $\Q^n$. This is given by a well-known  result on semimartingales under a change of measure,
Theorem \ref{theo-A} (Appendix \ref{sec-A}). It states that
$(X_t,\Q^n)_{t\in[0,T]}$   is a semimartingale with decomposition
\begin{multline*}
X_{t\wedge\tau_n}=X_0+\int_0^tI_{\{s\le\tau_n\}}\Big[a_s(X_{s-})+b_s(X_{s-})\sigma_s(X_{s-})
\\
+\int_\mathbb{R}h_s(X_{s-},z)\varphi_s(X_{s-},z)K(dz)\Big]ds
+\widetilde{\mathcal{M}}^{c,n}_t+\widetilde{\mathcal{M}}^{d,n}_t,
\end{multline*}
with   continuous $\widetilde{\mathcal{M}}^{c,n}_t$ and purely discontinuous
$\widetilde{\mathcal{M}}^{d,n}_t$
square
integrable martingales having predictable quadratic variations:
\begin{equation*}
\begin{aligned}
\langle
\widetilde{\mathcal{M}}^{c,n}\rangle_t&=\int_0^tI_{\{s\le\tau_n\}}b^2_s,X_{s-})ds
\\
\langle \widetilde{\mathcal{M}}^{d,n}\rangle_t
&=\int_0^t\int_\mathbb{R}I_{\{s\le\tau_n\}}
h^2_s(X_{s-},z)[1+\varphi_s(X_{s-},z)]K(dz)ds.
\end{aligned}
\end{equation*}
By  It\^o's formula we obtain
\begin{multline*}
X^2_{t\wedge\tau_n} =X^2_0+\int_0^tI_{\{s\le\tau_n\}}2X_{s-}\Big[a_s(X_{s-})+b_s(X_{s-})\sigma_s(X_{s-})
\nonumber\\
\quad
+\int_\mathbb{R}h_s(X_{s-},z)\varphi_sX_{s-},z)K(dz)\Big]ds+\langle
\widetilde{\mathcal{M}}^{c,n}\rangle_t
+[\widetilde{\mathcal{M}}^{d,n},\widetilde{\mathcal{M}}^{d,n}]_t
\\
\quad
+\int_0^tI_{\{s\le\tau_n\}}2X_{s-}d\widetilde{\mathcal{M}}^{c,n}_s+\int_0^tI_{\{s\le\tau_n\}}
2X_{s-}d\widetilde{\mathcal{M}}^{d,n}_s,
\end{multline*}
where $[\mathcal{M}^{d,n},\mathcal{M}^{d,n}]_t$ is the quadratic variation of $\mathcal{M}^{d,n}_t$.
Therefore
\begin{align}\label{eq:ups1}
\widetilde{\E}^nX^2_{t\wedge\tau_n}&=\widetilde{\E}^nX^2_0+
\int_0^t\widetilde{\E}^n I_{\{s\le\tau_n\}}2X_{s-}\Big[a_s(X_{s-})
+b_s(X_{s-})\sigma_s(X_{s-})
\nonumber\\
&\quad +\int_\mathbb{R}h_s(X_{s-},z)\varphi_s(X_{s-},z)K(dz)\Big]ds+
\widetilde{\E}^n\langle \widetilde{\mathcal{M}}^{c,n}\rangle_t
+\widetilde{\E}^n\langle \widetilde{\mathcal{M}}^{d,n}\rangle_t,
\end{align}
where $[\mathcal{M}^{d,n},\mathcal{M}^{d,n}]_t$ is a quadratic variation of the martingale
$\mathcal{M}^{d,n}_t$.
In view of
\begin{equation*}
\begin{aligned}
\widetilde{\E}^n\langle
\widetilde{\mathcal{M}}^{c,n}\rangle_t&=\widetilde{\E}^n\int_0^tI_{\{s\le\tau_n\}}b^2_s(X_{s-})ds
\\
\widetilde{\E}^n\langle \widetilde{\mathcal{M}}^{d,n}\rangle_t
&=\widetilde{\E}^n\int_0^t\int_\mathbb{R}I_{\{s\le\tau_n\}}
h^2_s(X_{s-},z)[1+\varphi_s(X_{s-},z)]K(dz)ds
\end{aligned}
\end{equation*}
\eqref{eq:ups1} can be presented as
$$
\widetilde{\E}^nX^2_{t\wedge\tau_n}=\widetilde{\E}^nX^2_0+\widetilde{\E}^n\int_0^tI_{\{s\le\tau_n\}}
\mathfrak{L}_s(X_{s-})ds,
$$
with $\mathfrak{L}_s(X_{s-})$ defined in  \eqref{eq:magl+}. Thus, we obtain the Gronwall-Bellman
inequality:
$
\widetilde{\E}^nX^2_{t\wedge\tau_n}\le \texttt{r}\int_0^t\big[1+\widetilde{\E}^nX^2_{s\wedge\tau_n}ds\big]
$
which implies the desired estimate
$$
\sup_n \int_0^T\widetilde{\E}^nX^2_{s\wedge\tau_n}ds\le e^{\texttt{r}T}-1.
$$

\subsection{Proof of Theorem \ref{theo-5.0.1}}
\label{sec-6.3}
Let stopping time $\tau_n$ (see \eqref{eq:taun}) is adapted to Theorem  \ref{theo-5.0.1}.
Then, by Lemma \ref{lem-unint} it suffices to verify the uniform integrability of family of random variables $\mathfrak{z}_{_{T\wedge\tau_n}}, n\ge 1$:
\begin{equation*}
\sup_n\widetilde{\E}^n\int_0^T\sup_{s'\le s}X^2_{s'\wedge\tau_n}ds<\infty.
\end{equation*}
In view of
$\E\mathfrak{z}_{_{T\wedge\tau_n}}=1$, the probability measure
$\Q^n$ is well defined:
$d\Q^n=\mathfrak{z}_{_{T\wedge\tau_n}}d\mathsf{P}$.
By Theorem \ref{theo-A} the process $(X_{t\wedge\tau_n},\Q^n)_{t\in[0,T]}$ is the semimartingale:
\begin{multline}
X_{t\wedge\tau_n}=X_0+\int_0^tI_{\{s\le\tau_n\}}\Big[a_s(X)+b_s(X)\sigma_s(X)
\\
+\int_\mathbb{R}h_s(X,z)\varphi_s(X,z)K(dz)\Big]ds
+\widetilde{\mathcal{M}}^{c,n}_t+\widetilde{\mathcal{M}}^{d,n}_t,
\label{eq:newQy}
\end{multline}
where $\widetilde{\mathcal{M}}^{c,n}_t$ a $\widetilde{\mathcal{M}}^{d,n}_t$
are continuous and purely discontinuous square integrable martingales with
\begin{equation*}
\begin{aligned}
\langle
\widetilde{\mathcal{M}}^{c,n}\rangle_t&=\int_0^tI_{\{s\le\tau_n\}}b^2_s(X)ds
\\
\langle \widetilde{\mathcal{M}}^{d,n}\rangle_t
&=\int_0^t\int_\mathbb{R}I_{\{s\le\tau_n\}}
h^2_s(X,z)[1+\varphi_s(X,z)]K(dz)ds.
\end{aligned}
\end{equation*}
This semimartingale enables to obtain the following estimate:
\begin{gather*}
\widetilde{\E}^n\sup_{t'\le t}|X_{t'\wedge\tau_n}|^2\le 4\Bigg[\widetilde{\E}^nX^2_0+
\widetilde{\E}^n\sup_{t'\le t}\big|\mathcal{M}^{c,n}_{t'}\big|^2
+\widetilde{\E}^n\sup_{t'\le t}\big|\mathcal{M}^{d,n}_{t'}\big|^2
\nonumber \\
+\underbrace{\widetilde{\E}^n\bigg(\int_0^tI_{\{s\le\tau_n\}}\Big|a_s(X)+b_s(X)\sigma_s(X)
+\int_\mathbb{R}h_s(X,z)\varphi_s(X,z)K(dz)\Big|ds\bigg)^2}_{:=J_t}\Bigg].
\end{gather*}
To this end, we evaluate each term in the right hand side of aforementioned inequality.
\begin{equation*}
\begin{array}{ll}
 4\widetilde{\E}^nX^2_0\le \texttt{r} \quad \text{by the assumption of theorem};
 \\ \\
 4\widetilde{\E}^n\sup_{t'\le t}\big|\mathcal{M}^{c,n}_{t'}\big|^2
\le 16\widetilde{\E}^n\langle \mathcal{M}^{c,n}\rangle_t
=16\widetilde{\E}^n\int_0^tI_{\{s\le\tau_n\}}b^2_s(X)ds \
\\
\mbox{}\hskip .85in\text{the maximal Doob inequality };
\\ \\
4\widetilde{\E}^n\sup_{t'\le t}\big|\mathcal{M}^{d,n}_{t'}\big|^2\le
16\widetilde{\E}^n\langle \mathcal{M}^{d,n}\rangle_t
\\
\hskip .3in=16\widetilde{\E}^n\int_0^t\int_\mathbb{R}I_{\{s\le\tau_n\}}
h^2_s(X,z)[1+\varphi_s(X,z)]K(dz)ds
\\
\mbox{}\hskip .85in\text{ìàêñèìàëüíîå íåðàâåíñòâî Äóáà };
\\ \\
4J_t\le 4t\E\int_0^tI_{\{s\le\tau_n\}}\big|a_s(X)+b_s(X)\sigma_s(X)
+\int_\mathbb{R}h_s(X,z)\varphi_s(X,z)K(dz)\big|^2ds
\\ \\
\le 12t\E\int_0^tI_{\{s\le\tau_n\}}\big[a^2_s(X)+b^2_s(X)\sigma^2_s(X)
+\big(\int_\mathbb{R}h_s(X,z)\varphi_s(X,z)K(dz)\big)^2\big]ds
\\ \\
\le 12t\E\int_0^tI_{\{s\le\tau_n\}}\big[a^2_s(X)+b^2_s(X)\sigma^2_s(X)
\\
\hskip 2in+\int_\mathbb{R}h^2_s(X,z)K(dz)\int_\mathbb{R}\varphi_s(X,z)K(dz)\big]ds
\\
\mbox{}\hskip .85in\text{the Cauchy-Schwarz inequality }.
\end{array}
\end{equation*}
These upper bounds imply the following inequality Âñÿ ñîâîêóïíîñòü
$
\widetilde{\E}^n\sup_{t'\le t}|X_{t'\wedge\tau_n}|^2\le \texttt{r}+\int_0^t
\widetilde{\E}^n\sup_{t'\le t}I_{\{s\le\tau_n\}}
\mathfrak{L}_s(X)ds,
$
where the operator $\mathfrak{L}_s(X)$ is defined in
\eqref{eq:LLL}, that is,
$
\widetilde{\E}^n\mathfrak{L}_{s\wedge\tau_n}(X)\le \texttt{r}\big[1+\widetilde{\E}^n\sup_{s'\le s}|X_{s'\wedge\tau_n}|^2\big].
$

Thus, we arrive at the Gronwall-Bellman inequality:
$$
\widetilde{\E}^n\sup_{t'\le t}|X_{t'\wedge\tau_n}|^2\le \texttt{r}\bigg[1+
\int_0^t\widetilde{\E}^n\sup_{s'\le s}|X_{s'\wedge\tau_n}|^2ds\bigg].
$$
So, the desired estimate
$
\sup_n\int_0^T\widetilde{\E}^n\sup_{s'\le s}|X_{s'\wedge\tau_n}|^2ds\le e^{\texttt{r}T}-1
$
holds true.

\section{\bf  Example when $X_t$ is a possibly  explosive Markov process }
\label{sec-5}

In this Section we consider a concrete model  (for a different example see
Andersen and Piterbarg \cite{AP}).

Let

\begin{equation*}
\begin{array}{ll}
X_t=X_0+\int_0^{t}a_s(X_{s-})ds
+\int_0^{t}b_s(X_{s-})dB_s
\\
\hskip 1.5in+\int_0^{t}\int_{\mathbb{R}}h_s(X_{s-},z)[\mu(ds,dz)-dsK(dz)]
\\ \\
\mathfrak{z}_t
=1+\int_0^t\mathfrak{z}_{s-}\Big[\sigma_s(X_{s-})dB_s
+\int_0^t\int_\mathbb{R}\varphi_s(X_{s-},z)[\mu(ds,dz)-dsK(dz)].
\end{array}
\end{equation*}
This example contains purely discontinuous martingales and so, it generalizes a model of Mijitovic-Urusov (see \cite{MiUr}).

Let the following conditions hold.
\begin{itemize}
  \item[\bf 1.] $a_s(x_{s-})\ge |x_{s-}|^\alpha, \alpha>3$
  \\
  \item [\bf 2.] $b^2_s(x_{s-})\le
  \begin{cases}
    \texttt{r} , & \alpha\in (3,4) \\
      \texttt{r}[1+x^2_{s-}] & \alpha>4
  \end{cases}
$
 \\
\item[\bf 3.] $\ h_s(x_{s-},z)\equiv z$
\\
\item[\bf 4.]$\sigma^2_s(x_{s-})\le \texttt{r}[1+x^2_{s-}]$
  \\
  \item[\bf 5.]$\ \varphi_s(x_{s-},z)\equiv |z|$
\\
\item[\bf 6.]$\int_\mathbb{R}[z^2+|z|^3]K(dz)<\infty$
 \\
  \item [\bf 7.]$0<X_0\le \texttt{r}$.
\end{itemize}
 \medskip
We assume   condition
  \begin{equation*}
\sigma^2_s(x_{s-})+\int_\mathbb{R}\varphi^2_s(x_{s-},z)K(dz)\le
 \big[1+x^2_{s-}\big]. \ s\in[0,T]
\end{equation*}
   However, conditions related to operators $L_s(x_{s-})$
and $\mathfrak{L}_s(x_{s-})$ fail.
Therefore an explosion of the process $X_t$ towards to $+\infty$ is possible. In a case of explsion
$\tau_n=\inf\{t:(\mathfrak{z}_t\vee X^2_t)\ge n\big\}$, $n\ge 1$ obeys a limit $\tau\le \infty$
as $n\to\infty$ with  $\mathsf{P}(\tau<\infty)>0$. So, in the case of explosion only
$
\E\mathfrak{z}_{_{T\wedge\tau}}=1
$
might be expected. The chance of explosion does not contradict the statement of Lemma
\ref{lem-unint}
$$
\sup_n\widetilde{\E}^n\int_0^T
X^2_{s\wedge\tau_n}ds<\infty.
$$
However, it is hard to check.
Therefore  we shall use
\begin{equation}\label{eq:+2-?}
\sup_n\widetilde{\E}^n\int_0^T
\frac{|X_{s\wedge\tau_n}|^\alpha}{1+|X_{s\wedge\tau_n}|^{\alpha-2}} ds<\infty
\end{equation}
instead
since for $\alpha>3$
$$
\sup_n\widetilde{\E}^n\int_0^T\frac{|X_{s\wedge\tau_n}|^\alpha}{1+|X_{s\wedge\tau_n}
|^{\alpha-2}} ds<\infty\Rightarrow
\sup_n\widetilde{\E}^n\int_0^TX^2_{s\wedge\tau_n}ds<\infty.
$$
We choose the function
\begin{gather*}
g_\alpha(u)=\int_0^u\frac{1}{1+|y|^{\alpha-2}}dy,\;\;u\in\mathbb{R}
\end{gather*}
with derivatives
\begin{gather*}
g'_\alpha (u)=\frac{1}{1+|u|^{\alpha-2}},
\quad\text{è}\quad
 g''_\alpha (u)=
-\frac{(\alpha-2)|u|^{\alpha-3}\sign(u)}{[1+|u|^{\alpha-2}]^2}.
\end{gather*}
Moreover  $a_s(x_{s-})\ge |x_{s-}|^\alpha$
è $\big|g''_\alpha(x_{s-})\big|\le \texttt{r}$. Áîëåå òîãî
\begin{equation}\label{eq:g''}
g'_\alpha (x_{s-})a_s(x_{s-})\ge \frac{|x_{s-}|^\alpha}{1+|x_{s-}|^{\alpha-2}},
\end{equation}
where the right hand side of this inequality is the integrand in
\eqref{eq:+2-?}. Since the process
$X_{t\wedge\tau_n}$, relative to the new measure $\Q^n$, does not explode  for any fixed  $n$, then by Theorem \ref{theo-A}, we have

\begin{equation*}
\begin{array}{ll}
X_{t\wedge\tau_n}=X_0+\int_0^tI_{\{s\le\tau_n\}}\Big[a_s(X_{s-})+b_s(X_{s-})\sigma_s(X_{s-})
\\
\hskip 2in +\int_\mathbb{R}z|z|K(dz)\Big]ds
+\widetilde{\mathcal{M}}^{c,n}_t+\widetilde{\mathcal{M}}^{d,n}_t,
\end{array}
\end{equation*}
where   $\widetilde{\mathcal{M}}^{c,n}_t$ and $\widetilde{\mathcal{M}}^{d,n}_t$
are continuous and purely discontinuous square integrable martingales with the predictable quadratic variations \begin{equation*}
\begin{array}{ll}
\langle
\widetilde{\mathcal{M}}^{c,n}\rangle_t=\int_0^tI_{\{s\le\tau_n\}}b^2_s(X_{s-})ds
\\
\langle \widetilde{\mathcal{M}}^{d,n}\rangle_t
=\int_0^t\int_\mathbb{R}I_{\{s\le\tau_n\}}
[z^2+|z|^3]K(dz)ds.
\end{array}
\end{equation*}
Now, by applying the It\^o formula to $g_\alpha (X_{t\wedge\tau_n})$ we obtain
\begin{align*}
&g_\alpha (X_{T\wedge\tau_n})=g_\alpha (X_0)
\\
&+\int_0^TI_{\{s\le\tau_n\}}g'_\alpha (X_{s-})\Big[a_s(X_{s-})
+\sigma_s(X_{s-})b_s(X_{s})
+\int_\mathbb{R}
|z|zK(dz)\Big]ds
\\
&+\int_0^tI_{\{s\le\tau_n\}}g'_\alpha (X_{s-})d\widetilde{\mathcal{M}}^{c,n}_s
+\int_0^tI_{\{s\le\tau_n\}}g'_\alpha (X_{s-})d\widetilde{\mathcal{M}}^{d,n}_s
\\
&+\frac{1}{2}\int_0^tI_{\{s\le\tau_n\}}g''_\alpha(X_{s-})d\langle \widetilde{\mathcal{M}}^{c,n}\rangle_s
\\
&
+\int_0^t\int_\mathbb{R}I_{\{s\le\tau_n\}}\Big[g_\alpha(X_{s-}+z)-g_\alpha(X_{s-})-g'_\alpha(X_{s-})z\Big]
\mu(ds,dz).
\end{align*}
Hence
 \begin{align*}
&\widetilde{\E}^ng_\alpha (X_{T\wedge\tau_n})=\widetilde{\E}^ng_\alpha (X_0)
\\
&+\widetilde{\E}^n\int_0^TI_{\{s\le\tau_n\}} g'_\alpha (X_{s-})a_s(X_{s})
ds
\\
&+\widetilde{\E}^n\int_0^TI_{\{s\le\tau_n\}}g'_\alpha(X_{s})\bigg[
\sigma_s(X_{s-})b_s(X_{s-})
+\int_\mathbb{R}
|z|zK(dz)\bigg]ds
\\
&+\frac{1}{2}\widetilde{\E}^n\int_0^TI_{\{s\le\tau_n\}}g''_\alpha (X_{s-})
b^2_s(X_s)ds
\\
&
+\widetilde{\E}^n
\int_0^T\int_\mathbb{R}I_{\{s\le\tau_n\}}\Big[g (X_{s-}+z)-g (X_{s-})-g' (X_{s-})z\Big]
\nu^n(ds,dz),
\end{align*}
where
$
\nu^n(ds,dz)=I_{\{s\le\tau_n\}}[1+|z|]dsK(dz)
$
is $\widetilde{\Q}^n$-compensator of integer-value random measure $I_{\{s\le \tau_n\}}\mu(ds,dz)$ (see Theorem \ref{theo-A}). So, in view of \eqref{eq:g''}, one can derive  the next upper bound:

 \begin{equation*}
\begin{aligned}
&\widetilde{\E}^n\int_0^TI_{\{s\le\tau_n\}}\frac{|X_{s-}|^\alpha}{1+|X_{s-}|^{\alpha-2}} ds\le \widetilde{\E}^n\int_0^TI_{\{s\le\tau_n\}}g'_\alpha (X_{s-})a_s(X_{s-})ds
\\
&\le
\widetilde{\E}^ng_\alpha (X_{T\wedge\tau_n})
+\widetilde{\E}^n\int_0^TI_{\{s\le\tau_n\}}
g'_\alpha(X_{s-})|\sigma_s(X)b_s(X-)|ds
\\
&+\widetilde{\E}^n\int_0^T\int_\mathbb{R}I_{\{s\le\tau_n\}}g' (X_{s-})
z^2K(dz)ds
\\
& \ +\frac{1}{2}\widetilde{\E}^n\int_0^TI_{\{s\le\tau_n\}}|g''_\alpha(X_{s-})|
b^2_s(X_{s-})ds
\\
&\ +\widetilde{\E}^n\int_0^T\int_\mathbb{R}I_{\{s\le\tau_n\}}
\big|g_\alpha(X_{s-}+z)-g_\alpha(X_{s-})-g'_\alpha(X_{s-})z\big|[1+|z|]K(dz)ds.
\end{aligned}
\end{equation*}
The absolute value of each    summand above is bounded  by a constant independent of   $n$:
\begin{itemize}
  \item $g_\alpha (X_{T\wedge\tau_n})\le \text{const.}$, $\alpha>3$
  \\
  \item
  $
g'_\alpha (X_{(t\wedge\tau_n)-}))\big|\sigma_t(X_{(t\wedge\tau_n)-})b_t(X_{(t\wedge\tau_n)-})\big|
\\ \\
\mbox{}\hskip 1in \le
\left\{
  \begin{array}{ll}
  \frac{\sqrt{1+X^2_{(t\wedge\tau_n)-})}}{1+|X^{\alpha-2}_{(t\wedge\tau_n)-}|}\le  \text{const.}  ,
  & \alpha\in(3,4) \\
   \frac{1+X^2_{(t\wedge\tau_n)-})}{1+|X^2_{(t\wedge\tau_n)-}|^{\alpha-2}}\le \text{const.} , & \alpha\ge 4
  \end{array}
\right.
$
 \\ \\
  \item $
\big|\int_\mathbb{R}g'_\alpha (X_{(t\wedge\tau_n)-})
\int_\mathbb{R}z^2K(dz)\big|\le \text{const.}
$
\\ \\
\item
$
\int_0^T|g''_\alpha (X_{(s\wedge\tau_n)-})|
b^2_s(X_{(s\wedge\tau_n)-})ds
\\ \\
\mbox{}\hskip .5in \le\int_0^T
\frac{(\alpha-2)|X_{(s\wedge\tau_n)-}|^{\alpha-3}}{[1+|X_{(s\wedge\tau_n)-}|^{\alpha-2}]^2}
[1+X^2_{(s\wedge\tau_n)-}]\int_\mathbb{R}z^2K(dz)ds\le \text{const.}
$
\\ \\
\item
$
  \int_0^T\int_\mathbb{R}
\big|g_\alpha (X_{s\wedge\tau_n}+z)-g_\alpha (X_{s\wedge\tau_n})-g'_\alpha(X_{(s\wedge\tau_n)-})z\big|\big[1+|z|\big]K(dz)ds
\\ \\
\mbox{}\hskip .5in\le \texttt{r}T\int_\mathbb{R}[z^2+|z|^3]K(dz).
$
\end{itemize}

Thus, \eqref{eq:+2-?} holds.

\section{\bf Extensions}
\label{sec-8+}

\subsection{Under $\pmb{\E\mathfrak{z}_{_T}<1}$ the Bene\^s condition may fail}
Let $X_t$ be Bessel process, $X_t=1+\int_0^t\frac{ds}{X_s}+B_t$.
By the It\^o
$$
\log(X_t)=-\int_0^t\frac{1}{2X^2_s}ds+\int_0^t\frac{dB_s}{X_s}.
$$
Consequently one may choose
$M_t=-\int_0^t\frac{dB_s}{X_{s}}$
and create Doleans-Dade process
$$
\mathfrak{z}_t=1-\int_0^t\mathfrak{z}_s\frac{dB_s}{X_s}.
$$
Assume there exists a positive time value $T$ such that
$\E \mathfrak{z}_{_T}=1$. Then also there exists a probability measure  $\Q\ll\mathsf{P}$ with density
$\frac{d\Q}{d\mathsf{P}}=\mathfrak{z}_{_T}$. So the Girsanov theorem enables present the process
$X_t$ (w.r.t. $\Q õ$) is:
$
X_t=1+\widetilde{B}_t
$
with $\Q$-Brownian motion  $\widetilde{B}_t$. Hence $X_t$ is Gaussian process ãàóññîâñêèé ïðîöåññ which cannot to be positive on $[0,T]$, a.s. So, $\E\mathfrak{z}_{_T}\ne 1$, i.e., $\E\mathfrak{z}_{_T}< 1$.
On the other hand, $\sigma^2(y)=\frac{1}{y^2}\not\le \texttt{r}\big[1+y^2]$.

\subsection{ $\pmb{X}$ - Vector diffusion  äèôôóçèÿ}
In view of diffusion setting we shall use notations:
Ïîñêîëüêó ìû èìååì äåëî ñ äèôôóçèîííûì ïðîöåññîì   îáîçíà÷åíèÿ
$x_s$, $X_s$ instead of $x_{s-}$, $X_{s-}$.
Let
\begin{equation*}
X_t=X_0+\int_0^ta_s(X_{s})ds+\int_0^tb_s(X_{s})dB_s,
\end{equation*}
where $a_s(x_s)$ is matrix-valued function, $b_s(x_s)$ are vector-valued function, $B_t$
is Brownian vector (column) motion process
with independent component (standard Brownian motions).
The Bene\^s condition is naturally compatible with vector case.
The norm in $\mathbb{L}^2$ and the inner product denote by  $\|\cdot\|$ and  $\lef\cdot\rig$  respectively,
the transposition symbol denote by $^*$ . Let $\sigma_s(x_s)$ be vector function (row),  such that,
 the process
$$
\mathfrak{z}_t=1+\int_0^t\mathfrak{z}_s\lef \sigma_s(x_s),dB_s\rig
$$
is well defined.

\begin{theorem}\label{theo-8.1.1}
Let $ \|X\|^2_0\le \texttt{r}$ è $(x_s)_{s\in \mathbb{R}_+}\in \mathbb{C}$.
 Let for any $(x_s)_{s\in[0,T]}\in \mathbb{C}$ the following property hold:

{\rm 1)} $\|\sigma_s(x_s)\|^2$

\vskip .05in
{\rm 2)} $L_s(x_s)=2\lef x_{s},a_s(x)\rig+\trace[b^{*}_s(x_s)b_s(x_s)]
$

\vskip .05in
{\rm 3)} $\mathfrak{L}_s(x_s)=2\lef x_{s},[a_s(x_s)+b^{*}_s(x_s)\sigma_s(x_s)]\rig+
\trace[b^{*}_s(x_s)b_s(x_s)].
$

If for any $s\le T$
$$
\|\sigma_s(x_s)\|^2+L_s(x_s)+\mathfrak{L}_s(x_s)\le \texttt{r}[1+\|x_s\|^2],
$$
then $\E\mathfrak{z}_{_T}=1, \ \forall \ T>0$.
 \end{theorem}
The proof of this theorem is similar to the proof in scalar setting, so, it is omitted.

\begin{example}\label{ex:8.1}
{\rm Let
$
X_t=X_0+\int_0^ta(X_s)ds +\int_0^tB(X_s)dB_s
$
and
$$
\mathfrak{z}_t=1+\int_0^t\mathfrak{z}_s\big[\theta(X_s)dW_s+\sigma(X_s)dB_s\big],
$$
where $W_t$ is Wiener process independent Brownian motion $B_t$. In spit of $X_t$ is scalar process, it is convenient to verify $\E\mathfrak{z}_{_T}=1$ by applying Theorem \ref{theo-8.1.1}.
Write
\begin{align*}
\begin{pmatrix}
  X_t \\
  0 \\
\end{pmatrix}
&=\begin{pmatrix}
  X_0\\
  0 \\
\end{pmatrix}+\int_0^t\begin{pmatrix}
 a(X_s) \\
  0\\
 \end{pmatrix}ds+\int_0^t\begin{pmatrix}
 b(X_s) & 0 \\
 0 & 0 \\
 \end{pmatrix}
 \begin{pmatrix}
   dB_s \\
   dW_s \\
 \end{pmatrix}
 \\
 \mathfrak{z}_t&=1+\int_0^t\mathfrak{z}_s\big[\theta(X_s)dW_s+\sigma(X_s)dB_s\big].
\end{align*}

If

1) $X^2_0\le \texttt{r}$

2) $\sigma^2(x)+\theta^2(x)\le \texttt{r}[1+x^2]$

3) $2xa(x)+b^2(x)
\le \texttt{r}[1+x^2]$

 4) $2x[a(x)+b(x)\sigma(x)]+
b^2(x)
\le \texttt{r}[1+x^2]$,

then $\E\mathfrak{z}_{_T}=1, \ \forall \ T>0$.

\medskip
Assume, $\sigma^2(x)\equiv 0$. In this setting, $X_t$ and $W_t$ are independent processes, so that, a well known result holds:
$
\E\mathfrak{z}_{_T}=\E\exp\big(\int_0^T\theta(X_s)dW_s-\frac{1}{2}\int_0^T\theta^2(X_s)ds\big)=1.
$
}
\end{example}

\subsection{Nonlinear version of Hitsuda's type model}

Let
\begin{equation*}
X_t=\int_0^t\int_0^sl(s,u)dB_uds
+B_t\quad\text{and}\quad \mathfrak{z}_t=1+\int_0^t\mathfrak{z}_s\sigma(X_s)dB_s,
\end{equation*}
where
$l(t,s)$ is ÿäðî Volterra kernel
$
\int_0^t\int_0^sl^2(s,u)duds<\infty,
$
and where
$\sigma(x)$ is nonlinear function.
In the case $\sigma(x)=x$ the condition $\iint\limits_{[0,T]^2}l^2(s,u)duds<\infty$
provides $\E\mathfrak{z}_{_T}=1$ (see \cite{Hits}).
If $\sigma(x)$ in nonlinear function (possible discontinuous), then, combining Hitsuda's  approach with Bene\^s condition, it is possible to obtain
\begin{theorem}\label{theo-6}
If

{\rm 1)} $\sigma^2(x)\le \texttt{r}\big[1+x^2\big], \ s\in[0,T]$

\vskip .1in
{\rm 2)} $\iint\limits_{[0,T]^2}l^2(t,s)dtds<\infty$,

then $\E\mathfrak{z}_{_T}=1$.
\end{theorem}
\begin{proof}
Formally, this model is not compatible with conditions of Theorem  \ref{theo-4.1.2} and Theorem \ref{theo-5.0.1}.
Nevertheless the choice of $\tau_n=\inf\big\{t:\big(\mathfrak{z}_{t}\vee X^2_t\big)\ge n\big\}$ and the
condition
\begin{equation}\label{ux}
\sup_n\widetilde{\E}^n\int_0^TX^2_{t\wedge\tau_n}dt<\infty,
\end{equation}
(see Lemma \ref{lem-unint}) provide  $\E\mathfrak{z}_{_T}=1$.
Condition 1) guaranties
$\E\mathfrak{z}_{_{T\wedge\tau_n}}=1$, that is, existence of probability measure $\Q^n$. Then, by
Girsanov theorem, a random process  $(B_{t\wedge\tau_n},\Q^n)$ can be presented as:
$$
B_{t\wedge\tau_n}=\int_0^tI_{\{s\le \tau_n\}}\sigma(X_s)ds+\widetilde{B}^n_t,
$$
with $\Q^n$ Brownian motion $\widetilde{B}^n_t$ stopped $\tau_n$.
Now, a random process $(X_{t\wedge\tau_n},\Q^n)_{t\in[0,T]}$ is semimartingale:
\begin{multline}\label{eq:neq}
X_{t\wedge\tau_n}=\int_0^{t}I_{\{s\wedge\tau_n\}}\int_0^{s}l(s,u)d\widetilde{B}^n_uds
\\
+\int_0^{t}I_{\{s\wedge\tau_n\}}\int_0^{s}l(s,u)\sigma(X_u)duds
\\
+ \int_0^tI_{\{s\le \tau_n\}}\sigma_s(X)ds+\widetilde{B}^n_t.
\end{multline}
This semimartingale will be applied in the proof of \eqref{ux}.
First of all we estimate $\widetilde{\E}^nX^2_{t\wedge\tau_n}$. In order to do that we evaluate
the expectation ($\widetilde{\E}^n$) of the each term square in the right hand side in
\eqref{eq:neq}.
Applying the Cauchy-Schwarz inequality, and the maximal Doob inequality for square integrable martingale,  and condition 1) we obtain,

\underline{First term:}
\begin{align*}
\ &\widetilde{\E}^n\bigg(\int_0^{t'}I_{\{s\le\tau_n\}}\int_0^{s}l(s,u)d\widetilde{B}^n_u
  ds\bigg)^2\le \widetilde{\E}^n\bigg(\int_0^{t}\sup_{s\le t}\bigg|\int_0^{s}l(s,u)d\widetilde{B}^n_u\bigg|
 ds \bigg)^2
 \\
 &\le t \widetilde{\E}^n\bigg(\int_0^{t}\sup_{s\le t}\bigg|\int_0^{s}l(s,u)d\widetilde{B}^n_u\bigg|^2ds\bigg)\le 4t^2\iint\limits_{[0,T]^2}l^2(s,u)dsds.
\end{align*}

\mbox{} \underline{Second term:}
\begin{multline*}
\widetilde{\E}^n\bigg(\int_0^{t}I_{\{s\le\tau_n\}}\int_0^{s}l(s,u)\sigma(X_u)duds\bigg)^2
\\
\hskip -1in\le t\int_0^t\widetilde{\E}^nI_{\{s\le\tau_n\}}\big|\int_0^{s}l(s,u)\sigma(X_u)du\big|^2ds
\\
\le t\int_0^t\int_0^{s}l^2(s,u)du\widetilde{\E}^n\int_0^{s\wedge\tau_n}\sigma^2(X_v)dvds
\\
\le t\int_0^t\int_0^{s}l^2(s,u)du\int_0^{s}\texttt{r}\big[1+\widetilde{\E}^n X^2_{v\wedge\tau_n}\big] dvds.
\end{multline*}

\mbox{} \underline{Third term:}
$
\widetilde{\E}^n\big(\int_0^tI_{\{s\le \tau_n\}}\sigma_s(X)ds\big)^2
\le t\int_0^{t}\texttt{r}\big[1+\widetilde{\E}^n X^2_{v\wedge\tau_n}\big] dvds.
$

\mbox{} \underline{Fourth term:}
$\widetilde{\E}^n\big(\widetilde{B}^n_t\big)^2\le \widetilde{\E}^n(t\wedge\tau_n)\le t$.

\medskip
\noindent
Obtained estimates enable arrive at the Gronwall-Bellman inequality:
$$
\widetilde{\E}^nX^2_{t\wedge\tau_n}
\le \texttt{r}\iint\limits_{[0,T]^2}l^2(t,s)dtds\bigg[1+\int_0^t\widetilde{\E}^n
X^2_{s\wedge\tau_n}ds\bigg]
$$
and, jointly with condition 2) of the theorem, to verify a validity  \eqref{ux}.
\end{proof}

\appendix
\section{\bf Generalized Girsanov theorem}\label{sec-A}

Let  $X$ solves equation \eqref{eq:sde} while stopping time $\tau_n$ is defined in  \eqref{eq:taun}.
Then
\begin{multline*}
X_{t\wedge\tau_n}=X_0+\int_0^tI_{\{s\le\tau_n\}}a_s(X)ds
\\
+\underbrace{\int_0^tI_{\{s\le\tau_n\}}b_s(X)dB_s}_{=\mathcal{M}_t^{c,n}}
+\underbrace{\int_0^t\int_\mathbb{R}I_{\{s\le\tau_n\}}h_s(X,z)[\mu(ds,dz)-dsK(dz)]}
_{=\mathcal{M}^{d,n}_t}.
\end{multline*}
Let the random variable  $\mathfrak{z}_{_{T\wedge\tau_n}}$ is defined in
\eqref{eq:34u}. Recall that $\E\mathfrak{z}_{_{T\wedge\tau_n}}=1$ and set a probability measure
$\Q^n$ having a density $\frac {d\Q^n}{d\mathsf{P}}=\mathfrak{z}_{_{T\wedge\tau_n}}$.

\begin{theorem}\label{theo-A}

{\rm (1)}
$
\nu^n(ds,dz)=I_{\{s\le\tau_n\}}[1+\varphi(s,X_{s-},z)]dsK(dz)
$
is $\Q^n$ - compensator of the integer-valued random measure  $I_{\{s\le\tau_n\}}\mu(ds,dz)$.

{\rm (2)} $(X_{t\wedge\tau_n},\mathcal{F}_t,\widetilde \Q^n)_{t\in[0,T]}$ is a semimartingale with decomposition{\rm :}
\begin{multline*}
X_{t\wedge\tau_n}=X_0+\int_0^tI_{\{s\le\tau_n\}}\Big[a_s(X)+b_s(X)\sigma_s(X)
\nonumber\\
+\int_\mathbb{R}h_s(X,z)\varphi_s(X,z)K(dz)\Big]ds
+\widetilde{\mathcal{M}}^{c,n}_t+\widetilde{\mathcal{M}}^{d,n}_t
\end{multline*}
in which $(\widetilde{\mathcal{M}}^{c,n}_t, \widetilde{\mathcal{M}}^{d,n}_t;\mathcal{F}_t,
\Q^n)_{t\in[0,T]}$ are continuous and purely discontinuous square integrable martingales.
Their predictable quadratic variations is defined below{\rm :}
\begin{equation}\label{eq:bubu}
\begin{aligned}
\langle
\widetilde{\mathcal{M}}^{c,n}\rangle_t&=\int_0^tI_{\{s\le\tau_n\}}b^2_s(X)ds
\\
\langle \widetilde{\mathcal{M}}^{d,n}\rangle_t
&=\int_0^t\int_\mathbb{R}I_{\{s\le\tau_n\}}
h^2_s(X,z)[1+\varphi_s(X,z)]K(dz)ds
\end{aligned}
\end{equation}
\end{theorem}
\begin{proof} (1)
  Let $\theta_n$ be stopping time, $\theta_n\le \tau_n$. Let $\varGamma\in \mathbb{R}\setminus \{0\}$ is a  measurable  set. Then
$$
\widetilde{\E}^n\int_0^T\int_\mathbb{R}I_{\{s\le\theta_n\}}I_{\{z\in\varGamma\}}\mu(ds,dz)
=\widetilde{\E}^n\int_0^T\int_\mathbb{R}I_{\{s\le\theta_n\}}I_{\{z\in\varGamma\}}\nu^n(ds,dz).
$$
On the other hand
\begin{align*}
&\widetilde{\E}^n\int_0^T\int_\mathbb{R}I_{\{s\le\theta_n\}}I_{\{z\in\varGamma\}}\mu(ds,dz)
=\E\mathfrak{z}_{_{T\wedge\theta_n}}\int_0^T\int_\mathbb{R}I_{\{s\le\theta_n\}}I_{\{z\in\varGamma\}}
\mu^n(ds,dz)
\\
&=\E\int_0^T\int_\mathbb{R}I_{\{s\le\theta_n\}}I_{\{z\in\varGamma\}}\mathfrak{z}_{_{s\wedge\theta_n}}
\mu^n(ds,dz)
\\&=\E\int_0^T\int_\mathbb{R}I_{\{s\le\theta_n\}}I_{\{z\in\varGamma\}}\mathfrak{z}_{_{(s\wedge\theta_n)-}}
[1+\varphi_s(X,z)]dsK(dz)
\\
&=\E\mathfrak{z}_{_{T\wedge\theta_n}}\int_0^T\int_\mathbb{R}I_{\{s\le\theta_n\}}I_{\{z\in\varGamma\}}
[1+\varphi(s,X_{s-},z)]dsK(dz)
\\
&=\widetilde{\E}^n\int_0^T\int_\mathbb{R}I_{\{s\le\theta_n\}}I_{\{z\in\varGamma\}}
[1+\varphi(s,X_{s-},z)]dsK(dz).
\end{align*}
Now, in view of an arbitrariness $\theta_n$ and $\varGamma$, statement (1) holds
(about additional details see \cite[\S 5, Ch.4 ]{LSMar} and \S 3.a - \S 3.d in \cite{JSh}).

\medskip
(2) To simplify notations denote
$$\begin{array}{c}
\mathcal{M}^{c,n}_t=\int_0^tI_{\{s\le \tau_n\}}b_s(X)dB_s,
\\ \\
\mathcal{M}^{d,n}_t=
\int_0^t\int_\mathbb{R}I_{\{s\le \tau_n\}}h_s(X,z)[\mu(ds,dz)-dsK(dz)].
  \end{array}
  $$
Taking into account that $\mathfrak{z}_{_{t\wedge\tau_n}}$ and $\mathcal{M}^{c,n}_t$,  $\mathcal{M}^{d,n}_t$ are square integrable martingales their quadratic characteristics are defined as:
\begin{align*}
&\langle \mathfrak{z}_{\cdot\wedge\tau_n}, \mathcal{M}^{c,n}\rangle_t
=\int_0^tI_{\{s\le\tau_n\}}\mathfrak{z}_{s-}\sigma_s(X)b_s(X))ds
\\
&[\mathfrak{z}_{\cdot\wedge\tau_n}, \mathcal{M}^{d,n}\rangle]_t
=\int_0^t\int_\mathbb{R}I_{\{s\le\tau_n\}}\mathfrak{z}_{s-}\varphi_s(X,z)h_s(X,z)\mu(ds,dz)
\end{align*}
Moreover
$\langle \mathfrak{z}_{\cdot\wedge\tau_n}\mathcal{M}^{d,n}\rangle_t$, being the compensator of
$[\mathfrak{z}_{\cdot\wedge\tau_n}, \mathcal{M}^{d,n}\rangle]_t$, obeys the following presentation
$$
\langle \mathfrak{z}_{\cdot\wedge\tau_n}\mathcal{M}^{d,n}\rangle_s=
\int_0^t\int_\mathbb{R}I_{\{s\le\tau_n\}}
\mathfrak{z}_{s-}\varphi_s(X,z)h_s(X,z)K(dz)ds.
$$
Next, by Theorem 2, \cite[\S 5, Ch.4 ]{LSMar},
\begin{equation*}
\begin{aligned}
\widetilde{\mathcal{M}}^{c,n}_t&=\mathcal{M}^{c,n}_t-
\int_0^t\mathfrak{z}^{-1}_{(s\wedge\tau_n)-}d\langle \mathfrak{z}_{\cdot\wedge\tau_n}
\mathcal{M}^{c,n}\rangle_s
\\
\widetilde{\mathcal{M}}^{d,n}_t&=\mathcal{M}^{d,n}_t
-\int_0^t\mathfrak{z}^{-1}_{(s\wedge\tau_n)-}d\langle \mathfrak{z}_{\cdot\wedge\tau_n},
\mathcal{M}^{d,n}\rangle_s
\end{aligned}
\end{equation*}
are continuous and purely discontinuous  $\Q^n$ - martingales.
In other words,
\newline $(\Q^n:\mathcal{M}^{c,n}_t; \mathcal{M}^{d,n}_t)$ processes obeys the presentations:
\begin{equation*}
\begin{aligned}
\mathcal{M}^{c,n}_t&=\int_0^tI_{\{s\le\tau_n\}}\sigma_s(X)b_s(X)ds
+\widetilde{\mathcal{M}}^{c,n}_t
\\
\mathcal{M}^{d,n}_t&=\int_0^t\int_\mathbb{R}I_{\{s\le\tau_n\}}\varphi_s(X,z)h_s(X,z)ds,
+\widetilde{\mathcal{M}}^{d,n}_t,
\end{aligned}
\end{equation*}
where $(\Q^n,\widetilde{\mathcal{M}}^{c,n}_t)$ is continuous martingale íåïðåðûâíûé and
$(\Q^n,\widetilde{\mathcal{M}}^{d,n}_t)$ is purely discontinuous martingale with
$\langle \widetilde{\mathcal{M}}^{c,n}\rangle_t$ and
$
\langle \widetilde{\mathcal{M}}^{d,n}\rangle_t
$
given  \eqref{eq:bubu}.
\end{proof}


\begin{thebibliography}{99}

\bibitem{AP} Andersen Leif B. and piterbarg Vladimir V. (2007) Moment explosions in stochastic volatility
models.
\emph{Finance Stoch.}  11. p. 29-50.
DOI 10.1007/s00780-006-0011-7.
\bibitem{BauNua} Baudoin, F. and Nualart, D. (2003) Equivalence of Volterra processes.
\emph{Stochastic Processes and their Applications} Vol. 107,  p. 327 – 350.

\bibitem{Benes} Benes, V.E. (1971)  Existence of optimal stochastic control laws
\emph{SIAM J. of Control}, 9 , 446-475

\bibitem{Cer} Cheridito, P. (2001) Mixed fractional Brownian motion. Bernoulli 7 , 913-934.

\bibitem{CFY}  Cheridito Patrick, Filipovi\'c Damir and Marc Yor. (2005)
Equivalent and absolutely continuous measure
changes for jump-diffusion processes. \emph{The Annals of Applied Probability}
Vol. 15, No. 3, p. 1713–1732.


\bibitem{Cox} Cox, J. C. (1997) The constant elasticity of variance option. Pricing
model. \emph{The Journal of Portfolio Management.} {\bf 23}, no. 2,
15-17.

\bibitem{Daw} Dawson, D. (1968). Equivalence of Markov processes. Trans. Amer. Math. Soc. 131
1–31. MR230375

\bibitem{DelShir} Delben, F. and Shirakawa, H. A Note of Option Pricing for
Constant Elasticity of Variance Model. available at:
www.math.ethz.ch

\bibitem{DoDade} Doleans-Dade C. (1970) Quelques applications de la formule de changement de variable pour les
semimartingales. \emph{Z. Wahrsch. verw.Geb.} Bd. 16, pp. 181-194.

\bibitem{Girs} Girsanov, I.V. (1960) On transforming a certan class of stochastic processes by absolutely continuous
substitution of measures. \emph{Theory Probab. Appl.} {\bf 5},
285-301.

\bibitem{Hits}  Hitsuda Masuyuki. (1968) Representation of Gaussian processes
equivalent to Wiener process. \emph{Osaka J. Math.} {\bf 5},
299-312.

\bibitem{ItoWat} It\^o, K. and Watanabe, S. (1965). Transformation of Markov processes by multiplicative
functionals. Ann. Inst. Fourier (Grenoble) 15 13–30. MR184282

\bibitem{JSh} Jacod, J. and Shiryaev, A. N. (1989) \emph{Limit theorems for stochastic processes}
Springer-Verlag Berlin, Heidelberg, New York, London, Paris, Tokio

\bibitem{KLSh} Kabanov, Ju.M.,  Liptser, R.S. and Shiryayev, A.N. (1979)
Absolute continuity and singularity of local absolutely continuous
probability distributions. I {\it Math. USSR Sbornik.} {\bf 35},
No. 5, pp. 631--680; II (1980) {\it Math. USSR Sbornik.} {\bf 36}, No. 1, pp. 31--58.

\bibitem{KSH} Kadota, T. and Shepp, L. (1970). Conditions for absolute continuity between a
certain pair of probability measures. Z. Wahrsch. Verw. Gebiete 16 250–260.
MR278344

\bibitem{KSh} Kallsen, J. and Shiryaev, A. N. (2002). The cumulant process and Esscher's change
of measure. Finance Stoch. 6 397–428. MR1932378

\bibitem
{142} Karatzas, I. and Shreve, S.E. (1991): Brownian Motion
and Stochastic Calculus. Springer-Verlag, New York Berlin
Heidelberg.


\bibitem{Kaz} Kazamaki, N. (1977) On a problem of Girsanov.// T\^ohoku Math. J., 29 , p. 597-
600.

\bibitem{ks} Kazamaki, n. and Sekiguchi, T.(1979)
On the transformation of some classes of martingales by a change of
law. \emph{T\^ohoku Math. Journ.} 31, p. 261-279.

\bibitem{Kry} Krylov, N.V. (8 May 2009) A simple proof of a result of A. Novikov.
arXiv:math/020713v2 [math.PR]

\bibitem{K} Kunita, H. (1969). Absolute continuity of Markov processes and generators. Nagoya
Math. J. 36 1–26. MR250387

\bibitem{K1} Kunita, H. (1976). Absolute continuity of Markov processes. S´eminaire de Probabilit
´es X. Lecture Notes in Math. 511 44–77. Springer, Berlin. MR438489

\bibitem{LM} L\'epingle, D. and M\'emin, J. (1978). Sur l'int´egrabilit´e uniforme des martingales
exponentielles. Z. Wahrsch. Verw. Gebiete 42 175–203. MR489492


\bibitem{LSh+} Liptser, R.S. and  Shiryaev A.N. (1972). On the absolute continuity of measures corresponding
to processes of diffusion type with respect to the Wiener
measures.// \emph{Izv. Acad. Nauk. SSSR. Ser. Mat.} 36, p. 847–889.



\bibitem{LSMar}  Liptser, R.Sh., Shiryayev, A.N. (1989) \emph{Theory of
Martingales.} Kluwer Acad. Publ.

\bibitem{LSI}  Liptser, R. Sh. and Shiryaev, A. N. (2000).
\emph{ Statistics of Random Processes {\rm I}}, 2nd ed.,  Springer,
Berlin - New York.

\bibitem{LSII}  LiptseR, R. Sh. and Shiryaev, A. N. (2000).
\emph{ Statistics of Random Processes {\rm II}}, 2nd ed.,  Springer,
Berlin - New York.

\bibitem{BeLip} Liptser R. Bene\^s condition for discontinuous
exponential martingale. http://arxiv.org/abs/0911.0641

\bibitem{MiUr} Mijitovic, A. and Urusov, M. On the Martingale Property of Certain Local Martingales: Criteria and Applications. arXiv:0905.3701v1 [math.PR]

\bibitem{Novikov} Novikov,  A.A. (1979) On the conditions of the uniform integrability of
the continuous nonnegative martingales.// Theory of Probability and
its Applications, 24, No. 4, p. 821-825.

\bibitem{UPR} Palmowski, Z. and Rolski, T. (2002). A technique for exponential change of measure
for Markov processes. Bernoulli 8, pp. 767-785. MR1963661

\bibitem{RevYor}   Revuz, D. and Yor M. (1991) \emph{Continuous martingales and Brownian moton}.
Springer-Verlag, Berlin.

\bibitem{Rud} Rydberg, T. (1997). A note on the existence of equivalent martingale measures in
a Markovian setting. Finance Stoch. 1, pp. 251-257.

\bibitem{JP} Picard, J. (2010) Representation formulae for the fractional
Brownian motion. \emph{arXiv:0912.3168v2 {\rm [}math.PR{\rm ]}}

\bibitem{ShAN} Shiryaev, A. N. (1999) Essentials of Stochastic Finance Facts, Models, Thoery.
Advanced Series on Statistical Science and Applied Probability, volume 3
Singapore: World Scientific.

\bibitem{Sin} Sin, Carlos A. (1998) Complications with Stochastic Volatility Models.
\emph{Advances in Applied Probability}, Vol. 30, No. 1 pp. 256-268
(http://www.jstor.org/stable/1427887)

\bibitem{WH+} Wong, B. and Heyde, C. C. (2004). On the martingale property of stochastic
exponentials. J. Appl. Probab. 41 654–664. MR2074814



\end{thebibliography}
\end{document}